\documentclass[11pt]{article} 
\usepackage{multicol} 
\usepackage{newlfont}
\usepackage{amsmath,amssymb}
\usepackage{tikz}

\def\rit{\mathbb{R}}
\def\zit{\mathbb{Z}}   
\def\nit{\mathbb{N}}

\def\ppit{\mathbb{P}} 
\def\qit{\mathbb{Q}} 
\def\cit{\mathbb{C}} 
\def\fit{\mathbb{F}}

\newcommand{\pf}{{\em Proof.~}}
\newcommand{\qed}{\hfill~~\mbox{$\Box$}}

\newenvironment{proof}{\smallskip \noindent \pf}{\qed \bigskip}

\newtheorem{theorem}{Theorem}[section]
\newtheorem{proposition}[theorem]{Proposition}
\newtheorem{definition}[theorem]{Definition}
\newtheorem{lemma}[theorem]{Lemma}
\newtheorem{corollary}[theorem]{Corollary}

\newtheorem{remark}[theorem]{Remark}
\newtheorem{example}[theorem]{Example}

\DeclareMathOperator{\card}{Card}
\DeclareMathOperator{\Inter}{Int}
\DeclareMathOperator{\vol}{vol}
\DeclareMathOperator{\boite}{Box}
\DeclareMathOperator{\Spec}{Spec}
\DeclareMathOperator{\alg}{alg}
\DeclareMathOperator{\st}{st}
\DeclareMathOperator{\geo}{geo}
\DeclareMathOperator{\Ehr}{Ehr}
\DeclareMathOperator{\orb}{orb}
\DeclareMathOperator{\Gr}{Gr}

\begin{document}

\title{\bf Global spectra, polytopes and stacky invariants}
\author{\sc Antoine Douai \thanks{Partially supported by the grant ANR-13-IS01-0001-01 of the Agence nationale de la recherche.
Mathematics Subject Classification 32S40, 14J33, 34M35, 14C40.
}\\
Universit\'e C\^ote d'Azur, CNRS, LJAD, France\\
Parc Valrose, F-06108 Nice Cedex 2, France \\
Email address: Antoine.DOUAI@unice.fr}

\maketitle

\begin{abstract} 
Given a convex polytope, we define its geometric spectrum, a stacky version of Batyrev's stringy E-functions,
and we prove a stacky version of a formula of Libgober and Wood about the E-polynomial of a smooth projective variety.
As an application, we get 
a closed formula for the variance of the geometric spectrum and
a Noether's formula for two dimensional Fano polytopes (polytopes whose vertices 
are primitive lattice points; a Fano polytope is not necessarily smooth). 
We also show that this geometric spectrum is equal to the algebraic spectrum
(the spectrum at infinity of a tame Laurent polynomial whose Newton polytope is the polytope alluded to).
This gives an explanation and some positive answers to 
Hertling's conjecture about the variance of the spectrum of tame regular functions.
\end{abstract}


\section{Introduction}

Let $X$ be a smooth projective variety of dimension $n$ with Hodge numbers $h^{p,q} (X)$. 
Using Hirzebruch-Riemann-Roch theorem, Libgober and Wood show 
in \cite{LW} the formula
\begin{equation}\label{eq:LWIntro}
\frac{d^2}{du^2} E (X; u,1)_{|u=1}=\frac{n(3n-5)}{12}c_n (X)+\frac{1}{6} c_1 (X) c_{n-1}(X)
\end{equation}
where $E( X; u,v)=\sum_{p,q} (-1)^{p+q}h^{p,q} (X) u^p v^q$ is the Hodge-Deligne polynomial of $X$.
By duality, it follows that
\begin{equation}\label{eq:VarianceDegreCohomologie}
\sum_{p,q} (-1)^{p+q} h^{p,q}(X) (p-\frac{n}{2})^2 =\frac{n}{12}c_n (X)+\frac{1}{6} c_1 (X) c_{n-1}(X).
\end{equation}
If $X$ is a $n$-dimensional projective variety with at most log-terminal singularities (we will focus in this paper on the toric case), 
Batyrev proves in \cite{Batyrev3} the ``stringy`` version of formula (\ref{eq:LWIntro})
\begin{equation}\label{eq:VarStringyIntro}
\frac{d^2}{du^2} E_{\st} (X; u,1)_{|u=1}=\frac{n(3n-5)}{12}e_{\st} (X)+\frac{1}{6} c_{\st}^{1,n}(X)
\end{equation}
where $E_{\st}$ is the stringy $E$-function of $X$, $e_{\st}$ is the stringy Euler number and 
$c_{\st}^{1,n}(X)$ is a stringy version of $c_1 (X) c_{n-1}(X)$.

On the singularity theory side (the B-side), the expected mirror partners of toric varieties are the Givental-Hori-Vafa models (see \cite{Giv}, \cite{HV}), in general a 
class of Laurent polynomials. One associates to such functions their {\em spectrum at infinity}, namely a sequence
$\alpha_{1},\cdots ,\alpha_{\mu}$ of rational numbers, suitable logarithms of the eigenvalues of the monodromy at infinity 
of the function involved (see \cite{Sab1}; the main features are recalled
in Section \ref{sec:SpecAlgebriquePolytope}). 
A specification of mirror symmetry is that the spectrum at infinity of a given Givental-Hori-Vafa model is related to the degrees of the 
(orbifold) cohomology groups of its mirror variety (orbifold). So one can expect a formula similar to
(\ref{eq:VarianceDegreCohomologie}) involving the spectrum at infinity of any regular function: the aim of this text 
is to look for such a counterpart. 

The key observation is that the spectrum at infinity of a Laurent polynomial can be described (under a tameness condition due to Kouchnirenko \cite{K}, see Section \ref{sec:SpecAlgebriquePolytope}) with the help of the Newton filtration of its Newton polytope. 
Since a polytope determines a stacky fan in the sense of \cite{BCS}, we are led to show a ``stacky'' version of formula (\ref{eq:LWIntro}). Given a Laurent polynomial $f$ with Newton polytope $P$, with global Milnor number $\mu$ (the number of critical points with multiplicities) and with spectrum at infinity $\alpha_1 ,\cdots ,\alpha_{\mu}$, the program is thus as follows:
\begin{itemize}
\item to construct a stacky version of the $E$-polynomial, the {\em geometric spectrum} of $P$: we define
$\Spec_P^{\geo} (z):=(z-1)^{n} \sum_{v\in N} z^{-\nu (v)}$ where $\nu$ is the Newton function of the polytope $P$, see
Section \ref{sec:SpecGeometriquePolytope}.
This geometric spectrum is closely related 
to the Ehrhart series and to the $\delta$-vector of the polytope $P$, more precisely to their twisted versions studied by 
 Stapledon \cite{Stapledon} and Musta\c{t}\u{a}-Payne \cite{MustataPayne}; it is also an orbifold Poincar\'e series (see Corollary \ref{coro:SpecGeoOrbifold}), thanks to the description of the orbifold cohomology given by Borisov, Chen and Smith \cite[Proposition 4.7]{BCS}, 
\item to show that this geometric spectrum is equal to the (generating function of the) spectrum at infinity of $f$,
and this is done by showing that both functions are Hilbert-Poincar\'e series of isomorphic graded rings (see Corollary \ref{coro:SpecGeoegalSpecAalg}):  
this gives the identification between the spectrum at infinity and the orbifold degrees, 
\item to show a ``stacky'' version of  (\ref{eq:LWIntro}), namely
\begin{equation}\nonumber
\frac{d^2}{dz^2} \Spec_P^{\geo} (z)_{|z=1}=\frac{n(3n-5)}{12}\mu_P +\frac{1}{6} \widehat{\mu}_P
\end{equation}
where $\mu_P$ is the normalized volume of $P$ (see equation (\ref{eq:MuVol})) and $\widehat{\mu}_P$ is a linear combination of intersection numbers, see Theorem \ref{theo:VarianceSpectreGeometrique}. 
\end{itemize}
 At the end, given a tame Laurent polynomial $f$, we get in Theorem \ref{theo:VarianceMiroirToriqueChampetre} the formula
\begin{equation}\label{eq:VarianceIntro}
\sum_{i=1}^{\mu}  (\alpha_i -\frac{n}{2})^2 =\frac{n}{12}\mu_P +\frac{1}{6}\widehat{\mu}_P 
\end{equation}
where $P$ is the Newton polytope of $f$ and $\alpha_1 ,\cdots ,\alpha_{\mu}$ is the spectrum at infinity of $f$.

In order to enlighten formula (\ref{eq:VarianceIntro}), assume that
$N=\zit^{2}$ and that $P$ is a full dimensional reflexive lattice polytope in $N_{\rit}$. Then we have the following well-known Noether's formula
\begin{equation}\label{eq:Noether}
12 =\mu_P +\mu_{P^{\circ}} 
\end{equation}
where $P^{\circ}$ is the polar polytope of $P$ and $\mu_{P}$ (resp. $\mu_{P^{\circ}}$) is the normalized volume of $P$ 
(resp. $P^{\circ}$). 
We show in Section \ref{sec:NoetherFanoPolytope} that $\widehat{\mu}_{P}=\mu_{P^{\circ}}$ if $P$ is a Fano polytope 
(that is if its vertices are primitive lattice points).
From formula (\ref{eq:VarianceIntro}), we then get 
\begin{equation}\label{eq:VarianceIntroDim2}
\sum_{i=1}^{\mu}  (\alpha_i -1)^2 =\frac{1}{6}\mu_P +\frac{1}{6}\mu_{P^{\circ}} 
\end{equation}
which is a generalization of formula (\ref{eq:Noether}): indeed, a reflexive polytope $P$ is Fano and its geometric spectrum satisfies $\sum_{i=1}^{\mu}  (\alpha_i -1)^2 =2$ (after a preprint version of this paper was written \cite{D0},
I have been informed that an analogous result was proposed independently by Batyrev and Schaller in \cite{BS}).

Last, notice that, because the mean value of $\alpha_1 ,\cdots , \alpha_{\mu}$ is $\frac{n}{2}$, we can use formula 
(\ref{eq:VarianceIntro}) in order to compute the variance 
$\frac{1}{\mu_P} \sum_{i=1}^{\mu}  (\alpha_i -\frac{n}{2})^2 $ of the spectrum at infinity of a Laurent polynomial $f$. For instance, assume that $\widehat{\mu}_P\geq 0$: it follows from equation (\ref{eq:VarianceIntro}) that
\begin{equation}\label{eq:ConjIntro}
\frac{1}{\mu_P}\sum_{i=1}^{\mu}(\alpha_i -\frac{n}{2})^2 \geq \frac{\alpha_{\max}-\alpha_{\min}}{12}
\end{equation}
where $\alpha_{\max}$ (resp. $\alpha_{\min}$) denotes the maximal (resp. minimal) spectral value of $f$ (indeed,  if $f$ is a Laurent polynomial we have $\alpha_{\max}-\alpha_{\min} =n$). This inequality is expected to be true for any tame regular function: this is the global version of Hertling's 
conjecture about the variance of the spectrum, see Section \ref{sec:Conj}. 
For instance, formula (\ref{eq:VarianceIntroDim2}) shows that this will be the case in the two dimensional case if the Newton polytope of $f$ is 
Fano.

This paper is organized as follows: in Section \ref{sec:PolytopeVarTor}  
we recall the basic facts on polytopes and toric varieties 
that we will use.
In Section \ref{sec:CahierChargesSpectre}, we define the spectrum of a 
polytope. The geometric spectrum is defined in Section \ref{sec:SpecGeometriquePolytope} 
and the algebraic spectrum is defined in Section \ref{sec:SpecAlgebriquePolytope}: both are compared
in Section \ref{sec:Comparaison}. 
The previous results are used in Section \ref{sec:ConjectureVarianceSpectre} in order to get formula 
(\ref{eq:VarianceIntro}). We show Noether's formula (\ref{eq:VarianceIntroDim2}) for Fano polytopes in Section \ref{sec:NoetherFanoPolytope}. Last, we use our results
in order to motivate (and to prove in some cases) the 
conjecture about the variance of the spectrum at infinity of a regular function in Section \ref{sec:Conj}. 

This text owes much to Batyrev's work \cite{Batyrev2}, \cite{Batyrev3}. The starting point was \cite[Remark 3.13]{Batyrev2} and its close resemblance with
Hertling's conjecture about the variance of the spectrum of an isolated singularity \cite{Her}: this link is previously alluded to in \cite{Her1}.

\section{Polytopes and toric varieties (framework)}

\label{sec:PolytopeVarTor}


\subsection{Polytopes and reflexive polytopes}

Let $N$ be the lattice $\zit^n$, let $M$ be its dual lattice and let $\langle \ ,\ \rangle$ be the pairing between $N_{\rit}=N\otimes_{\zit}\rit$ and $M_{\rit}=M\otimes_{\zit}\rit$.
A full dimensional lattice polytope $P\subset N_{\rit}$ is
the convex hull of a finite set of $N$ such that $\dim P=n$. 
If $P$ is a full dimensional lattice polytope containing the origin in its interior,
there exists, for each facet (face of dimension $n-1$) $F$ of $P$, $u_F \in M_{\qit}$ such that 
\begin{equation}\label{eq:PresentationFacette}
P\subset \{n\in N_{\rit},\ \langle u_F ,n\rangle \leq 1 \}\ 
\mbox{and}\ F=P\cap \{n\in N_{\rit},\ \langle u_F ,n\rangle = 1 \}.
\end{equation}
This gives the hyperplane presentation  
\begin{equation}\label{eq:PresentationPolytope}
P=\cap_F \{n\in N_{\rit},\ \langle u_F ,n\rangle \leq 1 \}.
\end{equation}
We define, for $v\in N_{\rit}$, $\nu_F (v):=\langle u_F ,v\rangle $ and 
$\nu (v):=\max_{F}\nu_{F}(v)$ where the maximum is taken over the facets of $P$.

\begin{definition}\label{def:FonctionSupport}
The function $\nu : N_{\rit}\rightarrow \rit$ is the Newton function of $P$.
\end{definition}

 Let $P$ be a full dimensional lattice polytope in $N_{\rit}$  containing the origin. The polytope
\begin{equation}\nonumber
P^{\circ}=\{m\in M_{\rit},\ \langle m,n \rangle\leq 1\ \mbox{for all}\ n\in P\}
\end{equation}
is the {\em polar polytope} of $P$.
The vertices of $P^{\circ}$ are in correspondence with the facets of $P$ {\em via} 
\begin{equation}\label{eq:CorrVerticePolar}
u_{F}\ \mbox{vertex of}\ P^{\circ} \leftrightarrow \  F =P\cap \{x\in N_{\rit}, \langle u_{F}, x\rangle =1\}. 
\end{equation}
A lattice  polytope $P$ is {\em reflexive} if it  
contains the origin and if
$P^{\circ}$ is a lattice polytope.

All the polytopes considered in this paper are full dimensional lattice polytopes containing the origin in their interior $\Inter P$.
For such a polytope $P$ we define its {\em normalized volume}
\begin{equation}\label{eq:MuVol}
\mu_P :=n! \vol (P)
\end{equation}
where the volume $\vol (P)$ is normalized such that the volume of the unit cube is equal to $1$.

\subsection{Ehrhart polynomial and Ehrhart series}
\label{sec:Ehrhart}

Let $Q$ be a full dimensional lattice polytope in $N_{\rit}$.
The function
$\ell \mapsto \Ehr_Q (\ell ):= \card ( (\ell Q )\cap N)$, $\ell\in\nit$,
is a  polynomial of degree $n:=\dim N_{\rit}$. This is the   
{\em Ehrhart polynomial} of $Q$.   
We have
\begin{equation}\label{eq:serie Ehrhart}
\sum_{m\geq 0}\Ehr_Q (m) z^m =\frac{\delta_0 +\delta_1 z +\cdots +\delta_n z^n}{(1-z)^{n+1}}
\end{equation}
where the $\delta_j$'s are non-negative integers \cite[Theorem 3.12]{BeckRobbins}: 
the generating function
\begin{equation}\nonumber 
F_Q (z):= \sum_{m\geq 0}\Ehr_Q (m) z^m
\end{equation} 
is the {\em Ehrhart series} of $Q$ and the vector
\begin{equation}\label{eq:deltavecteur}
\delta :=( \delta_0 ,\cdots ,\delta_n )\in\nit^{n+1}
\end{equation}
is the {\em $\delta$-vector} of $Q$.
The $\delta$-vector gives a characterization of reflexive polytopes:
the polytope $Q$ is reflexive if and only if $\delta_i =\delta_{n-i}$ for $i=0,\cdots ,n$, see for instance \cite[Theorem 4.6]{BeckRobbins}.

\subsection{Toric varieties}
\label{sec:VarPolFano}

Let $\Delta$ be a fan in $\nit_{\rit}$ and let $\Delta (i)$ be the set of its cones of dimension $i$.   
The {\em rays} of $\Delta$ are its one-dimensional cones. 
Let $X := X_{\Delta}$ be the toric variety of the fan 
$\Delta$: $X$ is {\em simplicial} if each cone of $\Delta$ is generated by independent vectors of $N_{\rit}$, {\em complete} if the support of its fan (the union of its cones) is $N_{\rit}$.

A full dimensional lattice polytope $Q$ in $M_{\rit}$ yields a toric variety $X_Q$, associated with the normal fan $\Sigma_Q$ of $Q$. Alternatively, if $P\subset N_{\rit}$ is a full dimensional lattice polytope containing 
the origin in its interior we get a complete fan $\Delta_P$ in $N_{\rit}$ by taking the cones over the proper faces of $P$ and we 
will denote by 
$X_{\Delta_P}$ the associated toric variety. 
 Both constructions are dual, see for instance \cite[Exercise 2.3.4]{CLS}.

Recall that a projective normal toric variety $X$  is {\em Fano} (resp. {\em weak Fano})
if the anticanonical divisor  $-K_X$ is $\qit$-Cartier and ample
(resp. nef and big). See \cite[Theorem 6.1.7]{CLS} for a characterization of  weak Fano toric varieties.
We will say that a full dimensional lattice polytope $P$ containing the origin in its interior is {\em Fano} if its vertices are primitive lattice points of $N$, {\em smooth Fano} if each of its facets has exactly $n$ vertices forming a basis of the lattice $N$.
It should be emphasized that a Fano polytope is not necessarily smooth.

{\em Otherwise stated, all toric varieties that we will consider are complete and simplicial.} 

\subsection{Stacky fans and orbifold cohomology}
\label{sec:EventailChampetre}

Let $\Delta$ be a complete simplicial fan and
let $\rho_1 ,\cdots ,\rho_r$ be its rays, generated respectively by the primitive vectors $v_1 ,\cdots ,v_r$ of $N$. Choose
$b_1 ,\cdots , b_r \in N$ whose images in $N_{\qit}$ generate the rays $\rho_1 ,\cdots ,\rho_r$: the data 
$\mathbf{\Delta} =(N, \Delta ,\{b_i\})$ is a  {\em stacky fan}, see \cite{BCS}.
In particular,
let $P$ be a lattice polytope containing the origin such that $\Delta :=\Delta_P$ is simplicial: there are  
$a_i$ such that $b_i :=a_i v_i \in \partial P\cap N$ and we will call the stacky fan 
$\mathbf{\Delta} =(N, \Delta ,\{b_i\})$ the {\em stacky fan of $P$}. 
One associates to this stacky fan a (separated) Deligne-Mumford stack ${\cal X}(\mathbf{\Delta})$, see \cite[Proposition 3.2]{BCS}.
We will denote by $H_{\orb}^{*} ( {\cal X}(\mathbf{\Delta}), \qit )$
its orbifold cohomology (with rational coefficients) and by 
$A_{\orb}^* ( {\cal X}(\mathbf{\Delta}))(=H_{\orb}^{2*} ( {\cal X}(\mathbf{\Delta}), \qit ))$ its orbifold Chow ring (with rational coefficients).

In this situation, we define, for a cone $\sigma\in \Delta$, 
\begin{itemize}
\item $N_{\sigma}$ the subgroup generated by $b_i$, $\rho_i \subseteq\sigma$,
\item $N(\sigma )=N/ N_{\sigma}$,
\item the fan $\Delta/ \sigma$ in $N(\sigma )_{\qit}$: this is the set 
$\{\tilde{\tau}=\tau +(N_{\sigma})_{\qit}, \ \sigma\subseteq \tau, \tau\in\Delta \}$,
\item $\boite (\sigma):=\{\sum_{\rho_{i}\subseteq\sigma}\lambda_i b_i,\ \lambda_i \in ]0,1[\}$.
\end{itemize}
By \cite[Proposition 4.7]{BCS} we have
\begin{equation}\label{eq:OrbiCohDecomposition}
 H_{\orb}^{2i} ({\cal X}(\mathbf{\Delta}), \qit )=\oplus_{\sigma\in\Delta }\oplus_{v\in \boite (\sigma)\cap N }H^{2(i-\nu (v))}(X_{\Delta/ \sigma }, \qit )
\end{equation}
where $\nu$ is the Newton function of $P$ (see Definition \ref{def:FonctionSupport}).

\subsection{Batyrev's stringy functions}
\label{sec:FonctionsFiliformes}

Let $X_{\Delta}$ be a normal $\qit$-Gorenstein toric variety and let
$\rho :Y\rightarrow X_{\Delta}$
be a toric resolution defined by a refinement $\Delta'$ of $\Delta$, see for instance \cite[Proposition 11.2.4]{CLS}. 
The irreducible components of the exceptional divisor of 
$\rho$ are in one-to-one correspondence with the primitive generators $v_{1}' ,\cdots ,v_{q}'$ of the rays of $\Delta' (1)$
of $Y$ that do not belong to $\Delta (1)$ and in the formula 
\begin{equation}\label{eq:ResolutionDiviseurCanonique}
 K_{Y}=\rho^{*}K_{X_{\Delta}}+\sum_{i=1}^{q}a_{i}D_{i}
\end{equation}
we have $a_{i}=\varphi (v'_{i})-1$ where $\varphi$ is the support function of the divisor $K_{X_{\Delta}}$, 
see \cite[Lemma 11.4.10]{CLS}.
In our toric situation we have 
$a_{i}>-1$ because $\varphi (v'_{i})>0$.

The $E$-polynomial of  a smooth variety  
$X$ is defined by
\begin{equation}\label{eq:Epolynome}
E (X, u,v):=\sum_{p,q=0}^n (-1)^{p+q} h^{p,q} (X) u^p v^q
\end{equation}
where the $h^{p,q}(X)$'s are the Hodge numbers of $X$. It is possible 
to extend this definition to singular spaces having log-terminal singularities
(and to get {\em stringy} invariants that extend topological invariants of smooth varieties) as follows:
let $\rho : Y\rightarrow X$ be a resolution of $X:=X_{\Delta}$ as above,
$I'=\{1,\cdots ,q\}$ and put, for any subset $J\subset I'$,
\begin{equation}\nonumber
D_J := \cap_{j\in J}D_j\ \mbox{if}\ J\neq \emptyset ,\ D_J := Y \ \mbox{if}\ J= \emptyset\ \mbox{and}\ D_J^{\circ}=D_J - \bigcup_{j\in I'-J}D_j.
\end{equation} 
The following definition is due to Batyrev \cite{Batyrev2} 
(we assume that the product over $\emptyset$ is $1$; recall that $a_i >-1$):

\begin{definition}\label{def:StringyFunction}
Let $X$ be a toric variety. 
The function 
\begin{equation}\label{eq:DefStringyFunction}
E_{\st}(X, u,v):=\sum_{J\subset I'}E( D_J^{\circ}, u,v)\prod_{j\in J}\frac{uv-1}{(uv)^{a_j +1}
-1}
\end{equation}
is the stringy $E$-function of $X$. The number
\begin{equation}\label{def:StringyeulerNumber}
e_{\st}(X):=\lim_{u,v\rightarrow 1}E_{\st}(X, u,v)
\end{equation}
is the stringy Euler number. 
\end{definition}

\noindent The stringy $E$-function can be defined using motivic integrals, see \cite{Batyrev2} and
\cite{Veys}. 
By \cite[Theorem 3.4]{Batyrev2}, $E_{\st}(X,u ,v)$ does not depend on the resolution.
In our setting, $E_{\st}$ depends only on the variable $z:=uv$, and we will write $E_{\st}(X, z)$ instead of $E_{\st}(X, u,v)$. 

In Section \ref{sec:InterpretationResolutionSing} we will use a modified version of the stringy $E$-function in order to compute the geometric spectrum of a polytope.

\section{The spectrum of a polytope}

\label{sec:CahierChargesSpectre}

Let $P$ be a full dimensional lattice polytope in $N_{\rit}$.
In this text, a spectrum $\Spec_P$ of $P$ is {\em a priori} an ordered sequence of rational numbers
$\alpha_{1}\leq \cdots\leq  \alpha_{\mu}$
that we will identify with the generating function
$\Spec_P (z) :=\sum_{i=1}^{\mu}z^{\alpha_{i}}$. The specifications are the following ($d(\alpha_i )$ denotes the multiplicity of $\alpha_i$
in the $\Spec_P$):

\begin{itemize}
\item {\em Rationality}: the $\alpha_i$'s are rational numbers, 
\item {\em Positivity}: the $\alpha_i$'s are non-negative numbers,
\item {\em Poincar\'e duality}: $\Spec_P (z)=z^n \Spec_P (z^{-1})$,
\item {\em Volume}: $\lim_{z\rightarrow 1} \Spec_P (z)= n! \vol(P)=:\mu_{P}$, 
\item {\em Normalisation}: $d(\alpha_1 )=1$.
\end{itemize}
In particular,
$\Spec_P$ is contained in $[0,n]$ and $\sum_{i=1}^{\mu}\alpha_{i}=\frac{n}{2}\mu_P $. 
Basic example: if $P$ is a smooth Fano polytope in $\rit^n$, the Poincar\'e polynomial $\sum_{i=1}^n b_{2i}(X_{\Delta_P}) z^i$ is a spectrum of $P$.

\section{Geometric spectrum of a polytope}
\label{sec:SpecGeometriquePolytope}

We define here the geometric spectrum of a polytope and we give several methods in order to compute it. Recall that the toric varieties considered here are assumed to be complete and simplicial.

\subsection{The geometric spectrum}

Let $P$ be a full dimensional lattice polytope in $N_{\rit}$, containing the origin in its interior. Recall the Newton function $\nu$ of $P$ of Definition \ref{def:FonctionSupport}.

\begin{definition}\label{def:specgeopoltope}
The function
\begin{equation}\nonumber
\Spec_P^{\geo} (z) :=(z-1)^{n} \sum_{v\in N} z^{-\nu (v)}
\end{equation}
is the geometric spectrum of  the polytope $P$.
The number 
$e_P :=\lim_{z\rightarrow 1} \Spec_P^{\geo}  (z)$
is the geometric Euler number of $P$.
\end{definition}

\noindent It will follow from Proposition \ref{prop:SpecGeohFunction} that 
$\Spec_P^{\geo}  (z)=\sum_{i=1}^{e_P} z^{\beta_{i}}$
for an ordered sequence  
$\beta_{1}\leq\cdots\leq \beta_{e_P}$ 
of non-negative rational numbers.

\subsection{Various interpretations}
We give three methods in order to compute $\Spec_P^{\geo}$, showing that it yields finally a spectrum of $P$
in the sense of Section \ref{sec:CahierChargesSpectre}. The first one and the third one  
are inspired by the works of Musta\c{t}\u{a}-Payne \cite{MustataPayne} and Stapledon \cite{Stapledon}. The second one is inspired by Batyrev's stringy $E$-functions.

\subsubsection{First interpretation: fundamental domains}

Let $P$ be a full dimensional lattice polytope in $N_{\rit}$, containing the origin in its interior and let $\Delta:=\Delta_P$ be the corresponding complete fan as in Section \ref{sec:VarPolFano}. We assume in this section that $\Delta$ is simplicial.
We identify each vertex of $P$ with an element $b_i \in N$.
If $\sigma \in \Delta (r)$ is generated by $b_1 ,\cdots ,b_r$, define
\begin{equation}\nonumber
\Box (\sigma ):=\{\sum_{i=1}^{r} q_i b_i,\ q_i \in [0,1[,\ i=1,\cdots ,r\}
\end{equation}
and
\begin{equation}\nonumber
\boite (\sigma ):=\{\sum_{i=1}^{r} q_i b_i,\ q_i \in ]0,1[,\ i=1,\cdots ,r\}.
\end{equation}

\begin{lemma}\label{lemma:SpecGeoBoxDimn}
We have
\begin{equation}\label{eq:DescriptionSpecGeon}
\Spec_P^{\geo}  (z)=\sum_{r=0}^{n}(z-1)^{n-r}\sum_{\sigma\in \Delta (r)}
\sum_{v\in \Box (\sigma )\cap N}z^{\nu (v)}
\end{equation}
and
$e_P =n!\vol(P)=:\mu_P$. 
\end{lemma}
\begin{proof}
Let $\sigma\in \Delta (r)$. A lattice element $v\in \stackrel{\circ}{\sigma}$ has one of the following decompositions: 
\begin{itemize}
\item $v =w +\sum_{i=1}^r \lambda_i b_i $ with
$w\in \boite (\sigma)\cap N$ and  $\lambda_i \in\nit$ for all $i$,
\item $v =w +\sum_{i=1}^r \lambda_i b_i $ with
$w\in \boite^c (\sigma )\cap N-\{0\}$, $\lambda_i \geq 0$ for all $i\geq 2$ and $\lambda_1 >0$ (up to renumbering),
\item $v= \sum_{i=1}^r \lambda_i b_i$ where $\lambda_i >0$ for all $i$
\end{itemize}
where 
$\boite^c (\sigma )$ is the complement of $\boite (\sigma )$ in $\Box (\sigma )$.
We get
\begin{equation}\label{eq:intermediaireBox}
(z-1)^r \sum_{v\in \stackrel{\circ}{\sigma}\cap N}z^{-\nu (v)}
=\sum_{v\in \boite (\sigma)\cap N }z^{r-\nu (v)} + 
\sum_{v\in \boite^c (\sigma )\cap N-\{0\}} z^{r-1-\nu (v)} +1
\end{equation} 
because 
\begin{itemize}
\item $\sum_{\lambda_1 ,\cdots, \lambda_r \geq 0}z^{-\nu (w)}z^{-\lambda_1}\cdots z^{-\lambda_r}=\frac{z^{r-\nu (w)}}{(z-1)^r}$
if $w\in \boite (\sigma)\cap N$,
\item $\sum_{\lambda_1 >0 ,\lambda_2 ,\cdots, \lambda_r \geq 0}z^{-\nu (w)}z^{-\lambda_1}\cdots z^{-\lambda_r}=\frac{z^{r-1-\nu (w)}}{(z-1)^r}$
if $w\in \boite^c (\sigma )\cap N - \{ 0\}$,
\item $\sum_{\lambda_1 ,\cdots, \lambda_r > 0}z^{-\lambda_1}\cdots z^{-\lambda_r}=\frac{1}{(z-1)^r}$
\end{itemize}
(and we use the fact that $\nu (b_i )=1$). Moreover,   
\begin{itemize}
\item $\alpha\in \nu (\boite (\sigma)):=\{\nu (v), v\in \boite (\sigma)\}$ 
if and only if
$r-\alpha\in \nu (\boite (\sigma))$,
\item $\alpha\in \nu (\boite^c (\sigma )):=\{\nu (v), v\in \boite^c (\sigma ) \}$ 
if and only if
$r-1-\alpha\in \nu (\boite^c (\sigma ))$
\end{itemize}
because $q_i \in ]0,1[$ if and only if $1-q_i \in ]0,1[$.
We then deduce from (\ref{eq:intermediaireBox}) that
\begin{equation}\label{eq:intermediaireBoxBis}
(z-1)^n \sum_{v\in \stackrel{\circ}{\sigma}\cap N}z^{-\nu (v)}
=(z-1)^{n-r}\sum_{v\in \Box (\sigma)\cap N }z^{\nu (v)} 
\end{equation} 
for any $\sigma\in\Delta (r)$. 
Equality (\ref{eq:DescriptionSpecGeon}) follows because the relative interiors of the cones of the complete fan 
 $\Delta$ give a partition of its support. 
For the assertion about the Euler number, notice that 
\begin{equation}\nonumber
\lim_{z\rightarrow 1} \Spec_P^{\geo}  (z)=\sum_{\sigma\in \Delta (n)}
\sum_{v\in \Box (\sigma )\cap N}1= n!\vol(P)
\end{equation}
because the normalized volume of $\sigma \cap \{v\in N_{\rit},\ \nu (v)\leq 1\}$ 
is equal to the number of lattice points in $\Box (\sigma )$.
\end{proof}

\begin{proposition}\label{prop:SpecGeohFunction}
Let $P$ be a full dimensional simplicial lattice polytope in $N_{\rit}$ containing the origin in its interior. Then
\begin{equation}\nonumber
\Spec_P^{\geo}  (z)=\sum_{\sigma\in \Delta }
\sum_{v\in \boite (\sigma )\cap N}h_{\sigma }(z) z^{\nu (v)}
\end{equation}
where $h_{\sigma}(z):=\sum_{\sigma \subseteq \tau} (z-1)^{n-\dim\tau}$.
\end{proposition}
\begin{proof} Follows from equation (\ref{eq:DescriptionSpecGeon}).
\end{proof}

\noindent It turns out that $h_{\sigma}(z)$ is the Hodge-Deligne polynomial of the orbit closure $V(\sigma )$ (as defined for instance in \cite[page 121]{CLS}) of the orbit $O(\sigma )$. Because $V(\sigma )$ is a toric variety, 
the coefficients of $h_{\sigma}(z)$ are non-negative integers (see for instance \cite[Lemma 2.4]{Stapledon} and the references therein) and we get
$\Spec_P^{\geo}  (z)=\sum_{i=1}^{\mu_P} z^{\beta_i}$
for a sequence $\beta_1 ,\cdots ,\beta_{\mu_P}$
of non-negative rational numbers. We will also call this sequence the {\em geometric spectrum of $P$}. 

\begin{corollary}\label{coro:SymetrieSpecGeo}
The geometric spectrum of $P$ satisfies $z^n \Spec^{\geo}_P (z^{-1})=\Spec^{\geo}_P (z)$. 
\end{corollary}
\begin{proof} Follows from Proposition \ref{prop:SpecGeohFunction}
because $z^{n-\dim\sigma}h_{\sigma}(z^{-1})=h_{\sigma}(z)$, see \cite[Lemma 2.4]{Stapledon},
and $\sum_{i}(1-m_{i})b_{i}\in \boite (\sigma )$ if $\sum_{i}m_{i}b_{i}\in \boite (\sigma )$.  
\end{proof}

\begin{corollary}\label{coro:SpecGeoOrbifold}
Let $P$ be a full dimensional simplicial lattice polytope in $N_{\rit}$ containing the origin in its interior and let
$\mathbf{\Delta} =(N, \Delta_P ,\{b_i\})$ be its stacky fan.
Then
\begin{equation}\nonumber
\Spec_{P}^{\geo}(z)=\sum_{\alpha\in\qit} \dim_{\cit}H^{2\alpha}_{\orb}({\cal X}(\mathbf{\Delta}),\cit) z^{\alpha}.
\end{equation}
The geometric spectrum of $P$ is the Hilbert-Poincar\'e series of the graded vector space $H^{2*}_{\orb}({\cal X}(\mathbf{\Delta}),\cit)$. 
\end{corollary}
\begin{proof} 
By formula (\ref{eq:OrbiCohDecomposition}), the orbifold degrees
are $\alpha =j+\nu (v)$ where $v\in \boite (\sigma )$ and $j=0,\cdots , n-\dim\sigma$. We thus get  
\begin{equation}\nonumber
 \sum_{\alpha}\dim_{\qit} H_{\orb}^{2\alpha}({\cal X}(\mathbf{\Delta}), \qit )z^{\alpha}=\sum_{\sigma\in\Delta}\sum_{v\in\boite (\sigma)\cap N }
\sum_{j=0}^{n-\dim \sigma}\dim_{\qit} H^{2j}(X_{\Delta/ \sigma}, \qit )z^{j} z^{\nu (v)}.
\end{equation}
Now, $\sum_{j=0}^{n-\dim \sigma}\dim_{\qit} H^{2j}(X_{\Delta/ \sigma}, \qit )z^{j}=h_{\sigma} (z)$
because the orbit closure
$V(\sigma )$ and the toric variety $X_{\Delta /\sigma}$ are isomorphic (see \cite[Proposition 3.2.7]{CLS}) and we get
\begin{equation}\nonumber
 \sum_{\alpha}\dim_{\qit} H_{\orb}^{2\alpha}({\cal X}(\mathbf{\Delta}), \qit )z^{\alpha}=
 \sum_{\sigma\in\Delta}\sum_{v\in\boite (\sigma)\cap N } h_{\sigma}(z) z^{\nu (v)}.
\end{equation}
The assertion then follows from Proposition \ref{prop:SpecGeohFunction}.
\end{proof}

To sum up, the geometric spectrum of a simplicial polytope is a spectrum in the sense of Section \ref{sec:CahierChargesSpectre}.
Rationality, positivity and the volume property are given by Lemma \ref{lemma:SpecGeoBoxDimn} and Proposition \ref{prop:SpecGeohFunction} and
symmetry (Poincar\'e duality) by Corollary \ref{coro:SymetrieSpecGeo}.

\subsubsection{Second interpretation: stacky $E$-function of a polytope (resolution of singularities)}
\label{sec:InterpretationResolutionSing}

Let $P$ be a full dimensional lattice polytope in $N_{\rit}$, containing the origin in its interior.
Let $\rho :Y\rightarrow X$ be a resolution of $X:=X_{\Delta_P}$ as in 
Section \ref{sec:FonctionsFiliformes} and let $\rho_1 ,\cdots ,\rho_r$ be the rays of $Y$, with
 primitive generators $v_1 ,\cdots ,v_r$ and associated divisors $D_1 ,\cdots, D_r$.
Put, for a subset $J\subset I:=\{1,\cdots ,r\}$,
$D_J := \cap_{j\in J}D_j$ if $ J\neq \emptyset$ and $D_J := Y$ if $J= \emptyset$ and
define 
\begin{equation}\label{eq:DefStringyFunctionChampetre}
E_{\st, P}(z):=\sum_{J\subset I}E( D_J , z)\prod_{j\in J}\frac{z-z^{\nu_j}}{z^{\nu_j}-1}
\end{equation} 
where $\nu_j =\nu (v_j)$ and $\nu$ 
is the Newton function of $P$ of
Definition \ref{def:FonctionSupport}.

\begin{proposition}\label{prop:SpecGeoEgalEstP}
 We have $\Spec_{P}^{\geo}(z)= E_{\st, P}(z)$. In particular, $E_{\st, P}(z)$ does not depend on the resolution $\rho$.
\end{proposition}
\begin{proof} 
Using the notations of Section \ref{sec:FonctionsFiliformes},
we have $E(D_J^{\circ}, z)=\sum_{J' \subset J}(-1)^{|J|-|J'|}E(D_{J'},z)$
and
\begin{equation}\nonumber
E_{\st, P}(z)=\sum_{J\subset I}E( D_J^{\circ} , z)\prod_{j\in J}\frac{z-1}{z^{\nu_j}-1}
\end{equation}
as in \cite[Proof of Theorem 3.7]{Batyrev2}.
Let $\sigma$ be a smooth cone of $\Delta'$, the fan of $Y$, generated by $v_{i_1},\cdots , v_{i_r}$ and $v\in \stackrel{\circ}{\sigma}$: we have
$v=a_1 v_{i_1}+\cdots +a_r v_{i_r}$ for $a_1 ,\cdots ,a_r >0$ and
$\nu (v)=a_1 \nu (v_{i_1} )+\cdots +a_r \nu (v_{i_r})$. Thus
$$\sum_{v\in\stackrel{\circ}{\sigma}\cap N}z^{-\nu (v)}= 
\frac{1}{z^{\nu (v_{i_1} )}-1}\cdots \frac{1}{z^{\nu (v_{i_r} )}-1}.$$
With these two observations in mind, the proof of the proposition is similar to the one of \cite[Theorem 4.3]{Batyrev2}.
\end{proof}

\begin{remark}
Applying Poincar\'e duality to the smooth subvarieties $D_J$, we get once again the symmetry relation
$z^n \Spec^{\geo}_P (z^{-1})=\Spec^{\geo}_P (z)$ of Corollary \ref{coro:SymetrieSpecGeo}.
\end{remark}

\subsubsection{Third interpretation: twisted $\delta$-vector}

Let $P$ be a full dimensional simplicial lattice polytope in $N_{\rit}$ containing the origin in its interior. 
Following \cite{Stapledon}, 
we define
$$F^0_P (z)=\sum_{m\geq 0 }\sum_{v\in mP\cap N}z^{\nu (v)-\lceil \nu (v)\rceil +m}.$$
This is a twisted version of the Ehrhart series $F_P (z)$ defined in Section \ref{sec:Ehrhart}. 

\begin{proposition}\label{prop:SpecGeodelta0} 
We have $\Spec_{P}^{\geo}(z)=(1-z)^{n+1} F^0_P (z)$.
 \end{proposition}
\begin{proof} 
Notice first that $v\in mP$ if and only if $\nu (v)\leq m$: this follows from the presentation (\ref{eq:PresentationPolytope}) and the definition of the Newton function $\nu$. We thus have  
$$F^0_P (z^{-1})
 =\sum_{m\geq 0 }\sum_{\nu (v)\leq m}z^{-\nu (v)+\lceil \nu (v)\rceil -m}
 =\sum_{v\in N}\sum_{\lceil \nu (v)\rceil \leq m}z^{-\nu (v)+\lceil \nu (v)\rceil -m}=
 \frac{1}{1-z^{-1}} \sum_{v\in N}z^{-\nu (v)}$$
 and this gives
 $(z-1)^n (1-z^{-1})F^0_P (z^{-1})=\Spec_{P}^{\geo}(z)$.
Thus
$$(1-z)^{n+1} F^0_P (z)=z^n \Spec_{P}^{\geo}(z^{-1})=\Spec_{P}^{\geo}(z)$$
where the last equality follows from Corollary \ref{coro:SymetrieSpecGeo}. 
\end{proof}

\begin{corollary}\label{coro:SpecGeo01} 
The part of the geometric spectrum contained in $[0,1[$ is the sequence 
$\nu (v),\  v\in \Inter P\cap N$. In particular, the multiplicity of $0$ in the geometric spectrum is equal to one.
Moreover the multiplicity of $1$ in $\Spec^{\geo}_{P}$ is equal to $\card (\partial P\cap N )-n$.
\end{corollary}
\begin{proof}
Scrutinization of the coefficients of $z^{a}$ for $a\leq 1$ in the formula of Proposition \ref{prop:SpecGeodelta0} (see also \cite[Lemma 3.13]{Stapledon}).
\end{proof}

If $P$ is reflexive, we have the following link between the $\delta$-vector of $P$ from Section \ref{sec:Ehrhart}
and its geometric spectrum, see also \cite{MustataPayne}:

\begin{corollary}\label{coro:SpecGeoReflexif}
Let $P$ be a reflexive full dimensional simplicial lattice polytope in $N_{\rit}$ containing the origin in its interior. 
Then $\Spec_{P}^{\geo}(z)=\delta_0 +\delta_1 z +\cdots +\delta_n z^n$
where $\delta =(\delta_0 ,\cdots ,\delta_n )$ is the $\delta$-vector of $P$.
\end{corollary}
\begin{proof} 
By (\ref{eq:CorrVerticePolar}), we have $\nu (v)\in\nit$ for all $v\in N$ because $P$ is reflexive and
we get $F^0_P (z)=F_P (z)$
where $F_P (z)$ is the Ehrhart series of $P$
of Section \ref{sec:Ehrhart} because
$$F_P (z)=\sum_{m\geq 0 }\card (mP\cap N) z^{m}=\sum_{m\geq 0 }\sum_{v\in mP\cap N}z^{m}.$$
By Proposition \ref{prop:SpecGeodelta0} we thus have
$\Spec_{P}^{\geo}(z)=(1-z)^{n+1}F_P (z)$
and we get the assertion using formula (\ref{eq:serie Ehrhart}). 
 \end{proof}

\section{Algebraic spectrum of a polytope}
\label{sec:SpecAlgebriquePolytope}

Singularity theory associates to a (tame) Laurent polynomial function $f$ a {\em spectrum at infinity}, see \cite{Sab1}. We recall its definition and its main properties in Section \ref{sec:SpecPolLaurent}. We can shift this notion
to the Newton polytope  $P$ of $f$ and we get in this way the {\em algebraic spectrum} of $P$. In order to motivate 
the next sections, we describe the Givental-Hori-Vafa models \cite{Giv}, \cite{HV} which are the expected mirror partners of toric varieties.

\subsection{Tameness (Kouchnirenko's framework)}
\label{sec:K}

We briefly recall the setting of \cite{K}. Let $f: (\cit^*)^n \rightarrow \cit$ be a Laurent polynomial,
$f (\underline{u})=\sum_{a\in\zit^n} c_a \underline{u}^a$ where $\underline{u}^a := u_1^{a_1}\cdots u_n^{a_n}$.
The Newton polytope $P$ of $f$ is the convex hull of the multi-indices $a$ such that $c_a \neq 0$.  
We say that $f$ is {\em convenient} if $P$ contains the origin in its interior,
{\em nondegenerate} if, for any face $F$ of $P$, the system
$$u_1 \frac{\partial f_F}{\partial u_1}=\cdots =u_n \frac{\partial f_F}{\partial u_n}=0$$
has no solution on $(\cit^* )^n$ where $f_F (\underline{u})=\sum_{a\in F\cap P} c_a \underline{u}^a$, 
the sum being taken over the multi-indices $a$ such that $c_a \neq 0$. 
A convenient and nondegenerate Laurent polynomial $f$ has only isolated critical points
and its global Milnor number $\mu_f$ (the number of critical points with multiplicities) is equal to the normalized volume $n! \vol (P)$, see \cite[Th\'eor\`eme III]{K}. 
Moreover, $f$ is tame in the sense that the set outside which $f$ 
is a locally trivial fibration is made from critical values of $f$, and these critical values belong to this set only because of the critical points at finite distance.

\subsection{Givental-Hori-Vafa models and mirror symmetry}
 \label{sec:ConstMiroirGen}

Let $N=\zit^n$ and let $M$ be the dual lattice. Let $\Delta$ be a complete and simplicial fan and let $v_1 ,\cdots ,v_r$ be the primitive generators of its rays. For the sake of simplicity, we assume here that the $v_i$'s generate the lattice $N$. Consider the exact sequence

\begin{equation}\nonumber
0\longrightarrow \zit^{r-n}\stackrel{\psi}{\longrightarrow} \zit^r \stackrel{\varphi}{\longrightarrow} \zit^n \longrightarrow 0
\end{equation}
where $\varphi (e_i) =v_i$ for $i=1,\cdots ,r$ ($(e_i )$ denotes the canonical basis of $\zit^r$) and $\psi$ describes the relations between the $v_i$'s.
Applying $Hom_{\zit} (--, \cit^* )$ to this exact sequence, we get 
\begin{equation}\nonumber
1\longrightarrow (\cit^*)^{n}\stackrel{}{\longrightarrow} (\cit^* )^r 
\stackrel{\pi}{\longrightarrow} (\cit^* )^{r-n} \longrightarrow 1
\end{equation}
where
\begin{equation}\label{eq:defpi}
\pi  (u_1 ,\cdots , u_r )= (q_1 ,\cdots , q_{r-n})=
(u_1^{a_{1,1}}\cdots u_r^{a_{r,1}},\cdots , u_1^{a_{1, r-n}}\cdots u_r^{a_{r, r-n}})
\end{equation}
and the integers $a_{i,j}$ satisfy  
$\sum_{j=1}^r a_{j, i} v_j =0$ for $i=1,\cdots , r-n$. The {\em Givental-Hori-Vafa model} of the toric variety $X_{\Delta}$ is the function  
\begin{equation}\nonumber
u_1 +\cdots + u_r \ \mbox{restricted to}\ U:= \pi^{-1}(q_1 ,\cdots , q_{r-n})
\end{equation}
We will denote it by $f_{X_{\Delta}}$. If $\Delta$ contains a smooth cone, the function $f_{X_{\Delta}}$ is easily described:

\begin{proposition}\label{prop:HVLaurent} Assume that $(v_{1},\cdots , v_{n})$ is the canonical basis of $N$.
Then $f_{X_{\Delta}}$ is the Laurent polynomial defined on $(\cit^*)^n$ by  
\begin{equation}\nonumber
f_{X_{\Delta}} (u_1 ,\cdots , u_n )= u_1 +\cdots + u_n + \sum_{i=n+1}^r q_i u_1^{v^i_1}\cdots u_n^{v^i_n}
\end{equation}
if $v_i = (v_1^i ,\cdots , v_n^i )\in\zit^{n}$ for $i=n+1,\cdots ,r$. 
\end{proposition}
\begin{proof}
We have $v_i = \sum_{j=1}^{n}v_j^i v_j$ for $i=n+1,\cdots ,r$ and the result follows from (\ref{eq:defpi}). 
\end{proof}

Above a convenient and nondegenerate Laurent polynomial $f$ we make grow a differential system 
(see \cite{DoSa1} and also \cite{DoMa} for explicit computations on weighted projective spaces) 
and we say that $f$ is a mirror partner of a variety $X$ if this differential system is isomorphic to the one associated with the (small quantum, orbifold) cohomology of $X$ (see \cite{D1}, \cite{RS}).
It is expected that the Givental-Hori-Vafa model $f_{X_{\Delta}}$ provides a mirror partner of the toric variety $X_{\Delta}$. 

If $f$ is the mirror partner of a smooth toric variety $X$  the following properties are in particular expected (non-exhaustive list):
\begin{itemize}
\item the Milnor number of $f$ is equal to the rank of the cohomology of $X$,
\item the spectrum at infinity of $f$ (see Section \ref{sec:SpecPolLaurent} below) is equal to half of the degrees of the cohomology groups of $X$,
\item multiplication by $f$ on its Jacobi ring yields the cup-product by $c_1 (X)$ on the cohomology algebra of $X$.
\end{itemize}
In particular, we have to restrict to (weak) Fano varieties: the rank of the cohomology of $X$ is the number of maximal cones while the Milnor number of $f$ is the normalized volume of its Newton polytope by \cite[Th\'eor\`eme III]{K}. See Example \ref{ex:SurfacesHirz} for a picture of the situation.

 In the singular simplicial case ({\em i.e} $X$ is an orbifold), cohomology should be replaced by orbifold cohomology, see Section \ref{sec:Comparaison}: in this situation, the rank of the orbifold cohomology is not a number of cones but a normalized volume, and the cohomology degrees are the orbifold degrees. 

\subsection{The spectrum at infinity of a tame Laurent polynomial}
\label{sec:SpecPolLaurent}

We assume in this section that $f$ is a convenient and nondegenerate Laurent polynomial, defined on  $U:=(\cit^{*})^n$. Let  $\mu$ be its global Milnor number. We use in this section the notations of \cite[2.c]{DoSa1}. 
Let $G$ be the Fourier-Laplace transform of the Gauss-Manin system of $f$, let $G_{0}$ be its Brieskorn lattice and let $V_{\bullet}$ be the $V$-filtration of $G$ at infinity, that is along $\theta^{-1}=0$. 
Because $f$ is convenient and nondegenerate, $G_0$ is a free $\cit [\theta ]$-module of rank $\mu$ and $G=\cit [\theta ,\theta^{-1}]\otimes G_0$, see \cite[Remark 4.8]{DoSa1}. 

From these data we get by projection a $V$-filtration on 
the $\mu$-dimensional vector space $\Omega_{f}:=\Omega^n (U)/df\wedge \Omega^{n-1}(U)=G_0 / \theta G_0$.
The {\em spectrum at infinity}  of $f$ is the spectrum of this filtration, 
that is the (ordered) sequence 
$\alpha_{1}, \alpha_{2},\cdots , \alpha_{\mu}$
of rational numbers with the following property: the frequency of $\alpha$ in the sequence is 
equal to $\dim_{\cit}\Gr^{V}_{\alpha}\Omega_{f}$. We will denote it by $\Spec_f$
and we will write
$\Spec_f (z) =\sum_{i=1}^{\mu}z^{\alpha_{i}}$. By \cite{Sab1} we have $\alpha_{i}\geq 0$ for all $i$ and $\Spec_f (z)=z^{n} \Spec_f (z^{-1})$.

If $f$ is convenient and nondegenerate, its spectrum at infinity can be computed using the 
Newton function of its Newton polytope: let us define
the Newton filtration $\mathcal{N}_{\bullet}$ on $\Omega^n (U)$ by
$$
\mathcal{N}_{\alpha}\Omega^{n}(U)=\{ \sum_{c\in N} a_c \omega_c ,\   \nu (\omega_c)\leq\alpha \ \mbox{for all}\ c\ \mbox{such that}\ a_c\neq 0\}
$$ 
where $\nu (\omega_c):=\nu (c)$ if $\omega_c =u_1^{c_1}\cdots u_n^{c_n} \frac{du_1\wedge\cdots\wedge du_n}{u_1\cdots u_n}$ and 
$ c=(c_1 ,\cdots , c_n )\in N$.
This filtration induces a filtration on $\Omega_f$ by projection and by \cite[Corollary 4.13]{DoSa1} the spectrum at infinity of $f$ is equal to the spectrum of this 
filtration.

\subsection{The algebraic spectrum of a polytope}
\label{sec:DefAlgSpec}

We define the algebraic spectrum $\Spec_P^{\alg}$ of a simplicial full dimensional lattice polytope $P$ 
 containing the origin in its interior to be the spectrum at infinity of the Laurent polynomial
 $f_P (\underline{u})=\sum_{b\in {\cal V}(P)}\underline{u}^b$ where ${\cal V}(P)$ denotes the set of the vertices of $P$.
Notice that $f_P$ is a convenient and nondegenerate Laurent polynomial and that its Milnor number is $\mu_{f_{P}}=\mu_{P}$: indeed, $f_P$ is convenient by definition because $P$ contains the origin in its interior and it is nondegenerate because of the simpliciality assumption. The assertion about the Milnor number then follows from \cite{K}. We will identify the algebraic spectrum with its generating function
$\Spec_P^{\alg} (z):=\sum_{i=1}^{\mu}z^{\alpha_{i}}$.

\begin{proposition}\label{prop:SpectreAlg01}
Let $P$ be a full dimensional simplicial lattice polytope in $N_{\rit}$ containing the origin in its interior. Then the part of the algebraic spectrum contained in $[0,1[$ is the sequence $\nu (v),\  v\in \Inter P\cap N$
where $\nu$ is the Newton function of $P$.
In particular, the multiplicity of $0$ in $\Spec_P^{\alg}$ is equal to one. 
 \end{proposition}
\begin{proof} The assertion for $\Spec_{f_P}$ follows from
\cite[Lemma 4.6]{DoSa1} and \cite[Example 4.17]{DoSa1}.
\end{proof}

In the two dimensional case, we deduce the following description of the algebraic spectrum:

\begin{proposition}\label{prop:SpecAlgDim2}
  Let $P$ be a full dimensional lattice polytope in $N_{\rit}=\rit^2$ containing the origin in its interior. Then
  \begin{equation}\nonumber
 \Spec_ P ^{\alg}(z) =(\card (\partial P\cap N)-2)z+
   \sum_{v\in \Inter P\cap N}(z^{\nu (v)}+z^{2-\nu (v)})
  \end{equation}
 where $\nu$ is the Newton function of $P$. 
 \end{proposition}
\begin{proof}  
 Let $f_P (\underline{u})=\sum_{b\in {\cal V}(P)}\underline{u}^b$ be as above.
From Proposition \ref{prop:SpectreAlg01}, and because of the symmetry property $\Spec_{f_P} (z)=z^{2} \Spec_{f_P}(z^{-1})$, we get
  \begin{equation}\nonumber
 \Spec_ {f_P} (z) =(\card (\partial P\cap N)-2)z+
   \sum_{v\in \Inter P\cap N}(z^{\nu (v)}+z^{2-\nu (v)})
  \end{equation}
The coefficient of $z$ is computed using Pick's formula (see for instance \cite[Theorem 2.8]{BeckRobbins}) because $\mu_{f_P} =2 \vol(P)$ by \cite[Th\'eor\`eme III]{K}. 
\end{proof}

In any dimension, we also have the following description for reflexive polytopes:

\begin{proposition}\label{prop:SpectrePolytopeLisse}
Let $P$ be a full dimensional reflexive simplicial polytope in $N_{\rit}=\rit^n$ containing the origin in its interior. Then: 
\begin{itemize}
\item $\Spec_P^{\alg} (z) =\sum_{i=0}^{n}d(i)z^{i}$ where $d(i)\in\nit$,
\item $d(i)=d(n-i)$ for $i=0,\cdots ,n$ with $d(0)=d(n)=1$,
\item $\sum_{i=0}^{n}d(i)=\mu_P$
\end{itemize}
\end{proposition}
\begin{proof} Because $P$ is reflexive, the Newton function takes integer values at the lattice points, 
see (\ref{eq:CorrVerticePolar}). This gives the first point because 
 $\Spec_P^{\alg}\subset [0,n]$. For the second one, use the symmetry and the fact that
 $0$ is in the spectrum with multiplicity one. 
\end{proof}

\begin{example} (Hirzebruch surfaces and their Givental-Hori-Vafa models)\\
\label{ex:SurfacesHirz}
Let $m$ be a positive integer.
The fan $\Delta_{\fit_{m}}$ of the Hirzebruch surface $\fit_m$  is the one whose rays are generated by the vectors 
$ v_{1}=(1,0)$, $v_{2}=(0,1)$, $v_{3}=(-1,m)$, $v_{4}=(0,-1)$,
see for instance \cite{Ful}. The surface
${\fit_m}$ is Fano if $m=1$, weak Fano if 
$m=2$. Its Givental-Hori-Vafa model is the Laurent polynomial 
\begin{equation}\nonumber
f_m (u_1 ,u_2 )=u_1 +u_2 +\frac{q_1}{u_2}+ q_2\frac{u_2^m}{u_1}.
\end{equation} 
We have 
\begin{enumerate}
\item $\mu_{f_1} =4$ and $\Spec_{f_1}(z)= 1+2z +z^2$,
\item $\mu_{f_2} =4$ if $q_2 \neq \frac{1}{4}$ and $\Spec_{f_2}(z)= 1+2z +z^2$,
\item $\mu_{f_m} =m+2$ if $m\geq 3$ and
$$\Spec_{f_m}(z)= 1+2z+z^2 +z^{\frac{1}{p}}+z^{\frac{2}{p}}+\cdots + z^{\frac{p-1}{p}}+
z^{2-\frac{1}{p}}+z^{ 2-\frac{2}{p}}+\cdots +z^{2-\frac{p-1}{p}}$$
if $m=2p$ and $p\geq 2$,
$$\Spec_{f_m}(z)=1+z+z^2 +z^{\frac{2}{m}}+z^{ \frac{4}{m}}+\cdots +z^{ \frac{2p}{m}}+
z^{2-\frac{2}{m}}+z^{2-\frac{4}{m}}+\cdots z^{ 2-\frac{2p}{m}}$$
if $m=2p+1$ and $p\geq 1$.
\end{enumerate}
Indeed,
for $m\neq 2$ we have 
$\mu_{f_m} =2! \vol(P)$, where $P$ is the Newton polytope of $f_m$,
because $f_m$ is convenient and nondegenerate for all non zero value of the parameters, see Section \ref{sec:K}.
For $m=2$, $f_2$  is nondegenerate if and only if             
$q_2 \neq \frac{1}{4}$ and the previous argument applies in this case
(if $q_2 =1/4$ the Milnor number is $2$). The spectrum is given by Proposition \ref{prop:SpecAlgDim2}.

The function $f_2$ is a genuine mirror partner of the surface $\fit_2$, see \cite{D}, \cite{RS}.
If $m\geq 3$, we have $\mu_{f_m} >4$ and the model $f_m$ has too many critical points.
\end{example}

\section{Geometric spectrum {\em vs} algebraic spectrum}
\label{sec:Comparaison}

We show in this section (Corollary \ref{coro:SpecGeoegalSpecAalg}) the equality
$\Spec_P^{\alg}(z)=\Spec_P^{\geo}  (z)$. The idea is to show that these functions are Hilbert-Poincar\'e series of isomorphic graded rings (Theorem \ref{theo:IsoGradedRings}). This is achieved using the description of the orbifold Chow ring given in \cite{BCS}.
Notice that in the two dimensional case, the expected equality follows immediately from 
Corollary \ref{coro:SpecGeo01}, Proposition \ref{prop:SpectreAlg01}
and the symmetry property.

\subsection{An isomorphism of graded rings}

Let $P$ be a full dimensional simplicial polytope containing the origin in its interior and
let $f_P (\underline{u})=\sum_{j=1}^r \underline{u}^{b_{j}}$ be the corresponding convenient and nondegenerate Laurent polynomial as in Section \ref{sec:DefAlgSpec}. 
In what follows, we will write $\underline{u}^c := u_1^{c_1}\cdots u_n^{c_n}$ if $c=(c_1 ,\cdots ,c_n )\in N$ and $K:=\qit [u_1, u_1^{-1},\cdots, u_n ,u_n^{-1}]$.
Let
\begin{equation}\nonumber
 \mathcal{A}_{f_{P}}=\frac{K}
 {\langle    u_1 \frac{\partial f_P}{\partial u_{1}},\cdots , u_n \frac{\partial f_P}{\partial u_{n}}       \rangle } 
\end{equation}
be the jacobian ring of $f_{P}$.
Define on $K$ the Newton filtration $\mathcal{N}_{\bullet}$ by
\begin{equation}\nonumber
\mathcal{N}_{\alpha}K= \{ \sum_{c\in N} a_c \underline{u}^c \in K,\   \nu (c)\leq\alpha \ \mbox{for all}\ c\ \mbox{such that}\ a_c\neq 0\} 
\end{equation}
where $\nu$ is the Newton function of $P$. The vector space
$\mathcal{N}_{<\alpha}K$ is defined similarly: replace the condition $\nu (c)\leq\alpha$ by $\nu (c)<\alpha$. 
This filtration induces by projection a filtration, also denoted by $\mathcal{N}_{\bullet}$ on $\mathcal{A}_{f_P}$, and we get the graded ring $\Gr^{\mathcal{N}}\mathcal{A}_{f_{P}}=\oplus_{\alpha} \mathcal{N}_{\alpha}\mathcal{A}_{f_{P}}/\mathcal{N}_{<\alpha}\mathcal{A}_{f_{P}}$ (see equation (\ref{eq:MultiplicationNewton}) below).

\begin{theorem}\label{theo:IsoGradedRings}
There is an isomorphism of graded rings
\begin{equation}\nonumber
A_{\orb}^* ({\cal X}(\mathbf{\Delta}))\cong \Gr^{\mathcal{N}}\mathcal{A}_{f_{P}}
\end{equation}
where $\mathbf{\Delta}$ is the stacky fan associated with $P$ as in Section \ref{sec:EventailChampetre}.
\end{theorem}
\begin{proof}
We first recall the setting of \cite[Section 5]{BCS}.
Let $\mathbf{\Delta} =(N, \Delta ,\{b_i\})$ be the stacky fan of $P$. We define its deformed group ring $\qit [N]^{\mathbf{\Delta}}$ as follows:
\begin{itemize}
\item as a $\qit$-vector space,
$\qit [N]^{\mathbf{\Delta}}=\oplus_{c\in N}\qit\ y^c$
where $y$ is a formal variable, 
\item the ring structure is given by 
$y^{c_1}.y^{c_2}=y^{c_1 +c_2}$ if $c_1$ and $c_2$ belong to the same cone, $y^{c_1}.y^{c_2}= 0$ otherwise,
\item the grading is defined as follows: if $c=\sum_{\rho_i \subseteq \sigma (c)} m_i b_i$ then $\deg (y^c )=\sum m_i \in\qit$ ($\sigma (c)$ is the minimal cone containing $c$). 
\end{itemize}
Observe that $\deg (y^c )=\nu (c)$ where $\nu$ is the Newton function of $P$.
Because $\Delta$ is simplicial, we have, by \cite[Theorem 1.1]{BCS},
an isomorphism of $\qit$-graded rings
\begin{equation}\label{eq:IsoOrbifoldChowRing}
\frac{\qit [N]^{\mathbf{\Delta}}}{\langle \sum_{i=1}^r \langle m, b_i \rangle y^{b_i} , \ m\in M\rangle}\longrightarrow  
\oplus_{v\in \Box (\Delta )} A^* ({\cal X}(\mathbf{\Delta /\sigma (v)}))[\deg (y^v )]
\end{equation}
where $\sigma (v)$ is the minimal cone containing $v$,
$\Box (\sigma):=\{\sum_{\rho_{i}\subset\sigma}\lambda_i b_i,\ \lambda_i \in [0,1[\}$ and
$\Box (\Delta )$ is the union of $\Box (\sigma )$ for all $n$-dimensional cones $\sigma\in\Delta$.

On the other side, we have
\begin{equation}\nonumber
 \mathcal{A}_{f_{P}}=\frac{K}
 {\langle \sum_{j=1}^{r}\langle m, b_{j}\rangle \underline{u}^{b_{j}}, \ m\in M\rangle } 
\end{equation}
 because
$u_i \frac{\partial f_P}{\partial u_{i}} =\sum_{j=1}^{r} \langle e_{i}^* , b_{j}\rangle \underline{u}^{b_{j}} $
 where $(e_{i}^*)$ is the dual basis of the canonical basis of $N$. 
Define the ring
\begin{equation}\nonumber
A_{f_{P}}=\frac{K^g}
 {\langle \sum_{j=1}^{r}\langle m, b_{j}\rangle \underline{u}^{b_{j}}, \ m\in M\rangle } 
\end{equation}
where $K^g =K$ as a vector space and the multiplication on $K^g$ is defined as follows: 
\begin{align}\nonumber
\underline{u}^{c_{1}}.\underline{u}^{c_{2}}=  \left\{ \begin{array}{ll}
\underline{u}^{c_{1}+c_{2}} & \mbox{if}\  
c_{1}\ \mbox{ and}\ c_{2}\ \mbox{ belong to the same cone}\ \sigma\ \mbox{of}\ \Delta_P,\\
0 & \mbox{otherwise}
\end{array}
\right .
\end{align}
\noindent Define a grading on $K^g$ by 
$\deg (\underline{u}^{c})=\nu (c)$.
Because $\nu (c_1 +c_2 )= \nu (c_1 )+\nu (c_2)$ if and only if $c_{1}$ and $c_{2}$ belong to the same cone $\sigma$ of $\Delta_P$ 
and because $\nu (b_j )=1$, the ring 
$A_{f_{P}}$ is graded. Moreover, because
\begin{align}\label{eq:MultiplicationNewton}
\underline{u}^{v}. \underline{u}^{w}\in  \left\{ \begin{array}{ll}
\mathcal{N}_{\alpha +\beta}K
 & \mbox{if}\  
\underline{u}^{v}\in \mathcal{N}_{\alpha}K\ \mbox{and}\ \underline{u}^{w}\in \mathcal{N}_{\beta}K ,\\
\mathcal{N}_{<\alpha +\beta}K & \mbox{if $v$ and $w$ do not belong to the same cone of $\Delta_P$},
\end{array}
\right .
\end{align}
the graded ring $A_{f_P}$ is isomorphic to $\Gr^{\mathcal{N}}\mathcal{A}_{f_{P}}$. 
Last, the rings
$\frac{\qit [N]^{\mathbf{\Delta}}}{\langle \sum_{i=1}^r \langle m, b_i \rangle y^{b_i} , \ m\in M\rangle}$
and $A_{f_P}$ are isomorphic, and
the theorem follows now from the isomorphism (\ref{eq:IsoOrbifoldChowRing}).
\end{proof}

\begin{corollary}\label{coro:SpecGeoegalSpecAalg}
 Assume that $P$ is a simplicial polytope containing the origin in its interior. Then
 $\Spec_P^{\alg}(z)=\Spec_P^{\geo}  (z)$.
\end{corollary}
\begin{proof}
As explained in Section \ref{sec:SpecPolLaurent},
$\Spec_P^{\alg} (z)$ is the Hilbert-Poincar\'e series of the graded vector space $\Gr^{\mathcal{N}} \mathcal{A}_{f_{P}}$.
Now, Theorem \ref{theo:IsoGradedRings} shows that the latter coincide with the Hilbert-Poincar\'e series of $A_{\orb}^* ({\cal X}(\mathbf{\Delta}))$  and we get
the assertion by Corollary \ref{coro:SpecGeoOrbifold}.
\end{proof}

\noindent This corollary can also be shown using \cite[Th\'eor\`eme 2.8]{K}.

\subsection{Application to weighted projective spaces}

\label{sec:EPP}

Let $(\lambda_{0},\cdots ,\lambda_{n})\in (\nit^*)^{n+1}$ be such that
$\gcd (\lambda_{0},\cdots ,\lambda_{n})=1$ and let $X$ be the weighted projective space  
$\ppit (\lambda_{0},\cdots ,\lambda_{n})$. 
{\em We assume from now on that} $\lambda_0 =1$.
The (stacky) fan of $X$ is given by: 
 \begin{enumerate}
\item the lattice $N=\zit^n$,
 \item the morphism $\beta :\zit^{n+1}\rightarrow N$,  which sends $e'_i$ to $e_i$ for $i=1,\cdots,n$ and $e'_0$ to $-\lambda_1 e_1 -\cdots -\lambda_n e_n$, where $(e'_i)$ is the canonical basis of $\zit^{n+1}$ and $(e_i )$ is the canonical basis of $N$,
 \item the fan $\Delta$, which is the complete and simplicial fan whose rays are generated by $b_{i}:=\beta (e'_i )$, $i=0,\cdots ,n$.
 \end{enumerate}
In this situation, we will call the convex hull $P$ of $b_{0},\cdots , b_{n}$ {\em the polytope of} $X$.
We have
$\mu_P =1+\lambda_{1}+\cdots +\lambda_{n}$ and $\mu_{P^{\circ}}=\frac{(1+\lambda_1 +\cdots +\lambda_n)^n}{\lambda_1 \cdots \lambda_n}$.

Let   
$$F:=\left\{\frac{\ell}{\lambda_{i}}|\, 0\leq\ell\leq \lambda_{i}-1,\ 0\leq i\leq n\right\}$$
and let $f_{1},\cdots , f_{k}$ be the elements of $F$ arranged by increasing order.
Define      
\begin{align}\nonumber
  S_{f_i}:=\{j|\ \lambda_{j}f_i \in\zit\}\subset \{0,\cdots ,n\}\ \mbox{and}\
d_{i}:=\card S_{f_{i}}
\end{align}
and let $c_{0},c_{1},\cdots , c_{\mu -1}$ be the sequence
$$\underbrace{f_{1},\cdots ,f_{1}}_{d_{1}},\underbrace{f_{2},\cdots ,f_{2}}_{d_{2}},\cdots ,\underbrace{f_{k},\cdots ,f_{k}}_{d_{k}}$$
arranged by increasing order. By \cite[Theorem 1]{DoSa2}, 
the spectrum at infinity of $f$
is the sequence
$\alpha_0 ,\alpha_1 ,\cdots , \alpha_{\mu -1}$ where
$$ \alpha_{k}:=k-\mu c_{k}\ \mbox{for}\ 
k=0,\cdots ,\mu -1.$$ 
\noindent Notice that this spectrum is integral if and only if the polytope $P$ of 
$X$ is reflexive.

We have
\begin{equation}\nonumber
f_P (u_1 ,\cdots , u_n )= u_1 +\cdots + u_n + \frac{1}{u_1^{\lambda_1}\cdots u_n^{\lambda_n}}.
\end{equation}
By Proposition \ref{prop:HVLaurent}, this is also the Givental-Hori-Vafa model of $X$. 
A mirror theorem is shown in \cite{DoMa}.

\begin{example}
\label{ex:SpectreGeoEPP} 

We test Corollary \ref{coro:SpecGeoegalSpecAalg} on weighted projective spaces.
The computation of the geometric spectrum is done using Proposition \ref{prop:SpecGeoEgalEstP}.
The rays are numbered clockwise.

\begin{enumerate}
\item Let $a$ be a positive integer and let $P$ be the convex hull of $(1,0)$, $(-1, -a)$ and $(0,1)$: this is the polytope of $\ppit (1,1,a)$.   
We consider the resolution obtained by adding 
the ray generated by $(0,-1)$. Using the notations of Proposition \ref{prop:SpecGeoEgalEstP},
 we have $\nu_{1}=1$, $\nu_2 =\frac{2}{a}$, $\nu_{3}=1$ and $\nu_{4}=1$ 
and we get 
\begin{equation}\nonumber
\Spec_{P}^{\geo} (z)=E(\fit_a ,z)+E(\ppit^1 ,z)\frac{z-z^{2/a }}{z^{2/a }-1}
=1+2z+z^2 +(1+z)(\frac{z-1}{z^{2/a}-1}-1)
\end{equation}
\begin{equation}\nonumber
=1+z+z^{2}+z^{2/a}+z^{4/a}+\cdots +z^{2(a-1)/a}=\Spec_P^{\alg} (z)
\end{equation}
where $\fit_{a}$ is the Hirzebruch surface.

\item Let $P$ be the convex hull of $(1,0)$, $(-2, -5)$ and $(0,1)$:
$P$ is the polytope of  $\ppit (1,2,5)$.       
We consider the resolution obtained by adding 
the rays generated by $(0,-1)$, $(-1,-3)$
and $(-1,-2)$. Using the notations of Proposition \ref{prop:SpecGeoEgalEstP},
 we have $\nu_1 =1$, $\nu_{2}=\frac{3}{5}$, $\nu_{3}=\frac{4}{5}$, $\nu_{4}=1$, $\nu_{5}=1$ and $\nu_{6}=1$ 
and we get
\begin{equation}\nonumber
\Spec_{P}^{\geo}(z)=z^2 +4z +1 +(z+1)\frac{z-z^{3/5}}{z^{3/5}-1}+(z+1)\frac{z-z^{4/5}}{z^{4/5}-1}
+\frac{z-z^{3/5}}{z^{3/5}-1}.\frac{z-z^{4/5}}{z^{4/5}-1}
\end{equation}
\begin{equation}\nonumber
=z^2 +2z +1 +z^{3/5} + z^{4/5}+z^{6/5}+z^{7/5}=\Spec_P^{\alg} (z).
\end{equation}

\item Let $\ell\in\nit^*$ and $P$ be the convex hull of $(1,0)$, $(-\ell , -\ell )$ and $(0,1)$:
$P$ is the polytope of $\ppit (1,\ell , \ell)$.
The variety $X_{\Delta_P}$ is $\ppit^2$, generated by the rays 
$v_0 =(-1, -1)$, $v_1=(1,0)$ and $v_2 =(0,1)$. Because $\nu (v_0 )=\frac{1}{\ell}$,
$\nu (v_1 )=1$ and $\nu (v_2 )=1$,
 we get    
\begin{equation}\nonumber
\Spec_{P}^{\geo}(z)=E(\ppit^2 ,z)+E(\ppit^1 ,z)\frac{z-z^{1/\ell }}{z^{1/\ell }-1}
=1+z+z^2 +(1+z)\frac{z-z^{1/\ell }}{z^{1/\ell }-1}
\end{equation}
\begin{equation}\nonumber
=z^2 +z +1 + z^{1/\ell}+\cdots +z^{(\ell -1)/\ell} +
z^{1+1/\ell}+\cdots + z^{1+(\ell -1)/\ell }=\Spec_P^{\alg} (z).
\end{equation}

\item Let $P$ be the convex hull of $(-2, -2, -2)$, $(1, 0, 0)$, $(0,1,0)$ and 
$(0,0,1)$: $P$ is the polytope of $\ppit (1, 2 , 2 , 2)$. We have $X_{\Delta_P}=\ppit^3$, generated by the rays 
$v_0 =(-1, -1,-1)$, $v_1=(1,0,0)$, $v_2 =(0, 1,0)$ and $v_3 =(0,0,1)$. Because $\nu (v_0 )=\frac{1}{2}$, we get 
$$\Spec_{P}^{\geo}(z) =E(\ppit^3 ,z)+E(\ppit^2 ,z)\frac{z-z^{1/2}}{z^{1/2}-1}$$
$$=z^3 +z^2 +z +1 + z^{1/2}+z^{3/2} +z^{5/2}=\Spec_{P}^{\alg}(z).$$

\end{enumerate}

\end{example}

\section{A formula for the variance of the spectra}
\label{sec:ConjectureVarianceSpectre}

We now prove formula (\ref{eq:VarianceIntro}) of the introduction.
We first show it for the geometric spectrum of a full dimensional simplicial lattice polytope.

\subsection{Libgober-Wood's formula for the spectra}

In order to get first a stacky version of the Libgober-Wood formula (\ref{eq:LWIntro}), 
we give the following definition, inspired by  Batyrev's stringy number $c_{\st}^{1, n-1} (X)$, see \cite[Definition 3.1]{Batyrev3}:

\begin{definition} 
Let $P$ be a full dimensional simplicial lattice polytope in $N$ containing the origin in its interior and
let $\rho :Y\rightarrow X$ be a resolution of $X:=X_{\Delta_P}$. We define the rational number
\begin{equation}\label{eq:ClasseBatyrevChampetre}
\widehat{\mu}_P:=c_1 (Y) c_{n-1} (Y)
+\sum_{J\subset I,\ J\neq\emptyset}c_1 (D_J )c_{n-|J|-1}(D_J ) 
\prod_{j\in J}\frac{1-\nu_j}{\nu_j}
\end{equation}
\begin{equation}\nonumber
-\sum_{J\subset I,\ J\neq\emptyset} (\sum_{j\in J} \nu_j )c_{n-|J|}(D_J ) \prod_{j\in J}\frac{1-\nu_j}{\nu_j }
\end{equation}
where the notations in the right hand term are the ones of Section \ref{sec:InterpretationResolutionSing} (convention:  $c_{r}(D_J )=0$ if $r<0$).
\end{definition}

\begin{remark}\label{remark:DiversCstP}
We have $\widehat{\mu}_P =c_1 (Y) c_{n-1} (Y)$ if $\nu_i =1$ for all $i$ (crepant resolutions) and
$\widehat{\mu}_P=c_1 (X) c_{n-1} (X)$ if $X$ is smooth.  
\end{remark}

\begin{theorem} 
\label{theo:VarianceSpectreGeometrique}
Let $P$ be a full dimensional simplicial lattice polytope in $N$ containing the origin in its interior and let
$\Spec_{P}^{\geo}(z)=\sum_{i=1}^{e_P}z^{\beta_i}$ be its geometric spectrum.
Then
\begin{equation}\label{eq:VarianceSpectreGeometrique}
\sum_{i=1}^{e_P}(\beta_i -\frac{n}{2})^2 =\frac{n}{12} \mu_P +\frac{1}{6} \widehat{\mu}_P
\end{equation}
where $\widehat{\mu}_P$ is defined by formula (\ref{eq:ClasseBatyrevChampetre}) and 
$\mu_P :=n!\vol (P)$.
\end{theorem}
\begin{proof} 
 Recall the stacky $E$-function 
$E_{\st, P}(z):=\sum_{J\subset I}E( D_J , z)\prod_{j\in J}\frac{z-z^{\nu_j}}{z^{\nu_j}-1}$, see formula (\ref{eq:DefStringyFunctionChampetre}).  
Then we have
\begin{equation}\label{eq:VarianceLWChampetre}
E''_{st,P} (1)=\frac{n(3n-5)}{12} e_{P}+\frac{1}{6}\widehat{\mu}_P 
\end{equation}
where $e_P$ is the geometric Euler number of $P$ of Definition \ref{def:specgeopoltope}.
The proof of this formula is a straightforward computation and is similar to the one of \cite[Theorem 3.8]{Batyrev3}: 
if $V$ is a smooth variety of dimension $n$, we have the Libgober-Wood formula
\begin{equation}\label{eq:LW}
E '' (V,1)=\frac{n(3n-5)}{12}c_n (V)+\frac{1}{6}c_1 (V) c_{n-1}(V)
\end{equation}
where $E$ is the $E$-polynomial of $V$, see \cite[Proposition 2.3]{LW}; 
in order to get (\ref{eq:VarianceLWChampetre}),
apply this formula
to the components $E(D_J ,z)$ of $E_{\st, P}(z)$ and use the equalities
\begin{equation}\nonumber
E(D_J ,1)=c_{n-|J|}(D_J ),\ 
\frac{d}{dz} (E(D_J ,z))_{|z=1}=\frac{n-|J|}{2}c_{n-|J|}(D_J )
\end{equation}
(the first one follows from the fact that the value at $z=1$ of the Poincar\'e polynomial is the Euler characteristic and we
get the second one using Poincar\'e duality for $D_J$)
and
\begin{equation}\nonumber
\frac{d}{dz}(\frac{z-z^{\nu}}{z^{\nu}-1})_{| z=1}=\frac{1-\nu}{2\nu},\
\frac{d^2}{dz^2}(\frac{z-z^{\nu}}{z^{\nu}-1})_{| z=1}=\frac{(\nu -1)(\nu +1)}{6\nu}
\end{equation}
if $\nu$ is a positive rational number. By Proposition \ref{prop:SpecGeoEgalEstP}, we have 
$\Spec^{\geo}_P (z)=E_{\st, P}(z)$ and we get
\begin{equation}\label{eq:Specf1} 
\frac{d^2}{dz^2}(\Spec_P^{\geo} (z))_{| z=1}=\frac{n(3n-5)}{12}e_P+\frac{1}{6} \widehat{\mu}_P.
\end{equation}
Because the geometric spectrum is symmetric with respect to $\frac{n}{2}$ (see Corollary \ref{coro:SymetrieSpecGeo}),
we have $\frac{d}{dz}(\Spec_P^{\geo}(z)) _{| z=1}= \frac{n}{2} e_P$ and we deduce that
\begin{equation}\label{eq:Specf1bis}
\frac{d^2}{dz^2}(\Spec_P^{\geo} (z))_{| z=1}
= \sum_{i=1}^{e_P}(\beta_i -\frac{n}{2})^2 +\frac{n(n-2)}{4} e_P.
\end{equation}
Now, formulas (\ref{eq:Specf1}) and (\ref{eq:Specf1bis}) give equality (\ref{eq:VarianceSpectreGeometrique}) because
$e_P =\mu_P$ by Lemma \ref{lemma:SpecGeoBoxDimn}.
\end{proof}

\begin{corollary} \label{coro:BLWindepedantRho}
The number $\widehat{\mu}_P$ does not depend on the resolution $\rho$. \qed
\end{corollary}

\noindent The version for singularities is straightforward:

\begin{theorem} 
\label{theo:VarianceMiroirToriqueChampetre}
Let $f$ be a convenient and nondegenerate Laurent polynomial with global Milnor number $\mu$ and 
spectrum at infinity $\Spec_{f} (z)=\sum_{i=1}^{\mu}z^{\alpha_i}$. 
Let $P$ be its Newton polytope (assumed to be simplicial). Then
\begin{equation}\label{eq:VarianceMiroirToriqueChampetre}
\sum_{i=1}^{\mu}(\alpha_i -\frac{n}{2})^2 =\frac{n}{12}\mu_P +\frac{1}{6} \widehat{\mu}_P
\end{equation}
where $\widehat{\mu}_P$ is defined by formula (\ref{eq:ClasseBatyrevChampetre}) and $\mu_P :=n!\vol (P)=\mu$.
\end{theorem}
\begin{proof}  By \cite{NS} we have $\Spec^{\alg}_P (z)=\Spec_{f} (z)$ and the assertion thus 
follows from
Theorem \ref{theo:VarianceSpectreGeometrique} and Corollary \ref{coro:SpecGeoegalSpecAalg}. Last, because $f$ is convenient and nondegenerate we have $\mu =\mu_P$ by \cite{K}.
\end{proof}

\subsection{The number $\widehat{\mu}_P$ in the two dimensional case}
\label{sec:TwoDimCase}

In this section we give an explicit formula for $\widehat{\mu}_P$ in the two dimensional case. 
Let $P$ be a full dimensional polytope in $N_{\rit}=\rit^2$ containing the origin, 
and let $\rho :Y\rightarrow X$ be a resolution of $X:=X_{\Delta_P}$ as in Section \ref{sec:InterpretationResolutionSing}.
We assume that the primitive generators $v_1 ,\cdots ,v_r$ of the rays of $Y$
are numbered clockwise and we consider indices as integers modulo $r$ 
so that $\nu_{r+1}:=\nu_1$ (recall that $\nu_i :=\nu (v_i )$ where $\nu$ is the Newton function of $P$).

\begin{proposition}\label{prop:MuPDim2}
We have 
\begin{equation}\nonumber
 \widehat{\mu}_P=c_1 ^2 (Y) -2r +\sum_{i=1}^r  (\frac{\nu_i}{\nu_{i+1}}+\frac{\nu_{i +1}}{\nu_{i}})=(\sum_{i=1}^r \nu_i D_i ) (\sum_{j=1}^r \frac{1}{\nu_{j}}D_j).
\end{equation}
\end{proposition}
\begin{proof} 
By definition, we have
\begin{equation}\label{eq:ClasseBatyrevChampetreDim2}
\widehat{\mu}_P:=c_1 (Y) c_{1} (Y)
+\sum_{J\subset I,\ J\neq\emptyset}c_1 (D_J )c_{1-|J|}(D_J ) 
\prod_{j\in J}\frac{1-\nu_j}{\nu_j}
\end{equation}
\begin{equation}\nonumber
-\sum_{J\subset I,\ J\neq\emptyset} (\sum_{j\in J} \nu_j )c_{2-|J|}(D_J ) \prod_{j\in J}\frac{1-\nu_j}{\nu_j }
\end{equation}
and thus
\begin{equation}\nonumber
 \widehat{\mu}_P=c_1 ^2 (Y) +
 2\sum_{i=1}^r \frac{(1-\nu_i )^2}{\nu_{i}}- \sum_{i=1}^r( \nu_i +\nu_{i+1} )\frac{(1-\nu_i )}{\nu_{i}}\frac{(1-\nu_{i+1})}{\nu_{i+1}}.
\end{equation}
It follows that
\begin{equation}\nonumber
\widehat{\mu}_P -c_1 ^2 (Y) 
=
\sum_{i=1}^r (\frac{1}{\nu_{i}}+\nu_i - \frac{1}{\nu_{i+1}}- \nu_{i+1} 
+\frac{\nu_i }{\nu_{i+1}} +\frac{\nu_{i+1}}{\nu_{i}}-2 )
\end{equation}
\begin{equation}\nonumber
=-2r + \sum_{i=1}^r (\frac{\nu_i }{\nu_{i+1}} +\frac{\nu_{i+1}}{\nu_{i}})
\end{equation}
and this gives the first equality. For the second one, notice that
\begin{equation}\nonumber
(\sum_{i=1}^r \nu_i D_i ) (\sum_{j=1}^r \frac{1}{\nu_{j}}D_j)
=\sum_{i=1}^r (D_i^2 +\frac{\nu_i}{\nu_{i+1}}D_i D_{i+1} 
+\frac{\nu_i}{\nu_{i-1}}D_i D_{i-1})
\end{equation}
\begin{equation} \nonumber
=\sum_{i=1}^r (D_i^2 +\frac{\nu_i}{\nu_{i+1}} 
+\frac{\nu_{i+1}}{\nu_{i}})
=c_1^2 (Y)-2r+\sum_{i=1}^r (\frac{\nu_i}{\nu_{i+1}} 
+\frac{\nu_{i+1}}{\nu_{i}})=\widehat{\mu}_P.
\end{equation}
\end{proof}

\begin{corollary}\label{coro:VarianceDim2Vraie}
Let $P$ be a full dimensional lattice polytope in $N_{\rit}=\rit^2$ containing the origin in its interior and let
$\Spec_{P}^{\geo}(z)=\sum_{i=1}^{e_P}z^{\beta_i}$ be its geometric spectrum.
Then
\begin{equation}\label{eq:VarianceMiroirToriqueChampetreDim2}
\sum_{i=1}^{\mu_P}(\beta_i -1)^2 =\frac{\mu_P}{6} +\frac{\widehat{\mu}_P}{6} 
\end{equation}
where $\widehat{\mu}_P\geq c_1 ^2 (Y)$ for any resolution $\rho : Y\rightarrow X_{\Delta_P}$.
\end{corollary}
\begin{proof} The equality follows from Theorem \ref{theo:VarianceSpectreGeometrique}.
By Proposition \ref{prop:MuPDim2} we have $\widehat{\mu}_P\geq c_1 ^2 (Y)$ because $\nu +\frac{1}{\nu}\geq 2$ for all real positive numbers $\nu$. 
\end{proof}

\subsection{Examples}
\label{ex:VarianceEPP}

\subsubsection{Weighted projective spaces}
We test Theorem \ref{theo:VarianceMiroirToriqueChampetre} on the weighted projective spaces
considered in Example \ref{ex:SpectreGeoEPP}. 
Let $P$ be the polytope of the weighted projective space $X=\ppit (1,\lambda_1 ,\cdots ,\lambda_n )$ and let $f(\underline{u})=\sum_{i=0}^n \underline{u}^{b_i}$ where the $b_i$'s are defined as in Section \ref{sec:EPP}. 
We will denote by $\sum_{i=1}^{\mu} z^{\alpha_{i}}$ the spectrum at infinity of $f$
and we will write $V(\alpha ):= \sum_{i=1}^{\mu}(\alpha_i -\frac{n}{2})^2$. 

\begin{itemize}
\item The polytope $P$ is Fano (see Section \ref{sec:VarPolFano}):
\begin{center}

\begin{tabular}{|c|c|c|c|c|} \hline\hline
$X$                &  $\mu$        &  $V(\alpha )$                         & $\mu n/12$      & $ \widehat{\mu}_P$  \\ \hline\hline
$\ppit (1,1,a)$        & $a+2$         &  $(2a^2 +6a+4)/6a$                    & $(a+2)/6$       &  $\frac{(a+2)^2}{a}$ \\ \hline
$\ppit (1,2,5)$        & $8$           & $12/5$                                & $4/3$           &  $32/5$\\ \hline

\end{tabular}

\end{center}

\noindent For $\ppit (1,1,a)$,
 the polytope $P$ is the convex hull of $(1,0)$, $(0,1)$ and
$(-1 ,-a )$ and we consider the resolution obtained by adding 
the ray generated by $(0,-1)$. 
We use Proposition \ref{prop:MuPDim2} in order to compute $\widehat{\mu}_P$, with 
$\nu_{1}=1$, $\nu_2 =\frac{2}{a}$, $\nu_{3}=1$, $\nu_{4}=1$ and $c_1^2 (Y)=8$. 

\noindent For $\ppit (1,2,5)$, 
 the polytope $P$ is the convex hull of $(1,0)$, $(0,1)$ and
$(-2 ,-5 )$ and  
we consider the resolution obtained by adding 
the rays generated by $(0,-1)$, $(-1,-3)$
and $(-1,-2)$. We use Proposition \ref{prop:MuPDim2} in order to compute $\widehat{\mu}_P$, with $\nu_1 =1$, $\nu_{2}=\frac{3}{5}$, $\nu_{3}=\frac{4}{5}$, $\nu_{4}=1$, $\nu_{5}=1$, $\nu_{6}=1$ and $c_1^2 (Y)=6$. 

\noindent Notice that in these examples we have
$\widehat{\mu}_P=\mu_{P^{\circ}}$ where $\mu_{P^{\circ}}$ is the volume of the polar polytope: this is not a coincidence, see Section 
\ref{sec:NoetherFanoPolytope} below.

\item The polytope $P$ is not Fano:

\begin{center}

\begin{tabular}{|c|c|c|c|c|} \hline\hline
$X$               &  $\mu$        &  $V(\alpha )$                         & $\mu n/12$      & $\widehat{\mu}_P$  \\ \hline\hline
$\ppit (1,\ell ,\ell)$ & $1+2\ell$     & $2+\frac{(\ell -1)(2\ell -1)}{3\ell}$ & $(2\ell +1)/6$  & $9+2\frac{(\ell -1)^2}{\ell}$\\ \hline
$\ppit (1,2,2,2)$      & $7$           & $7$                                   & $7/4$   & $63/2$\\ \hline

\end{tabular}

\end{center}

\vspace{5mm}

\noindent For $\ppit (1,\ell , \ell)$, $\ell\geq 2$,  the polytope $P$ is the convex hull of $(1,0)$, $(0,1)$ and
$(-\ell ,-\ell )$. 
Formula (\ref{eq:ClasseBatyrevChampetre}) gives 
\begin{equation}\nonumber
\widehat{\mu}_P =c_1 (\ppit^2) c_{1} (\ppit^2 )
+c_1 (\ppit^1 ) (\ell -1) 
-\frac{1}{\ell} c_{1}(\ppit^1 ) (\ell -1).
\end{equation}
Notice that $\widehat{\mu}_P\neq \mu_{P^{\circ}}$.

\noindent For $\ppit (1, 2 , 2 , 2)$, $P$ is the convex hull of $(-2, -2, -2)$, $(1, 0, 0)$, $(0,1,0)$ and 
$(0,0,1)$. 
Formula (\ref{eq:ClasseBatyrevChampetre}) gives
\begin{equation}\nonumber
\widehat{\mu}_P =c_1 (\ppit^3) c_{2} (\ppit^3 )
+c_1 (\ppit^2 ) c_1 (\ppit^2 ) 
-\frac{1}{2} c_{2}(\ppit^2 ). 
\end{equation}
\end{itemize}

\subsubsection{Miscellaneous}
 In order to complete the panaorama, let us now consider somewhat different situations:
\begin{itemize} 
\item let $P_{1,2,2}$ be the polytope with vertices $b_1 =(1,0)$, $b_2 =(0,2)$ and $b_3 =(-2 ,-2)$. 
\item Let $P_{\ell ,\ell ,\ell }$ be the polytope with vertices $b_1 =(\ell ,0)$, $b_2 =(0,\ell )$ and $b_3 =(-\ell ,-\ell )$
where $\ell$ is a positive integer. 
\end{itemize}

\noindent We have the following table:

\begin{center}

\begin{tabular}{|c|c|c|c|c|} \hline\hline
              &  $\mu$        &  $V(\alpha )$                         & $\mu n/12$      & $ \widehat{\mu}_P$  \\ \hline\hline
$P_ {1,2,2}$        & $8$           &  $3$                               & $4/3$           &  $10$ \\ \hline
$P_{\ell ,\ell ,\ell }$        & $3\ell^2$         &  $(\ell^2 +3)/2$                    & $\ell^2 /2$       &  $9$ \\ \hline
\end{tabular}

\end{center}

\noindent This agrees with formula (\ref{eq:VarianceMiroirToriqueChampetre})

\section{A Noether's formula for two dimensional Fano polytopes}

\label{sec:NoetherFanoPolytope}

In this section, we still focus on the two dimensional case: $P$ is full dimensional polytope in $N_{\rit}=\rit^2$.

Assume that $P$ is a reflexive polytope. Then $X_{\Delta_P}$ has a crepant resolution and 
$\widehat{\mu}_P=c_1^2 (Y)$ by Remark \ref{remark:DiversCstP}. The anticanonical divisor of $Y$ is nef 
and $c_1^2 (Y)=\mu_{P^{\circ}}$ by \cite[Theorem 13.4.3]{CLS}, so that finally
$\widehat{\mu}_P=\mu_{P^{\circ}}$. By Corollary \ref{coro:SpecGeoReflexif}, the geometric spectrum of $P$ 
satisfies 
$\sum_{i=1}^{\mu}(\beta_i -1)^2 =2$ 
and we thus get from Corollary \ref{coro:VarianceDim2Vraie} the well-known formula
 \begin{equation}\label{eq:Noetherformula}
12=\mu_P +\mu_{P^{\circ}}
\end{equation}
for a reflexive polytope $P$.

We have the following generalization of equation (\ref{eq:Noetherformula}) for the Fano polytopes 
defined in Section \ref{sec:VarPolFano} (a reflexive polytope is Fano):

\begin{theorem}\label{theo:VarianceMuPolytopePolaire}
Assume that $P$ is a Fano polytope in $\rit^2$ and let $\Spec_P^{\geo} (z)=\sum_{i=1}^{\mu_P} z^{\beta_i}$ be its geometric spectrum. Then
\begin{equation}\label{eq:VarianceMuPolytopePolaireEq}
\sum_{i=1}^{\mu_P}(\beta_i -1)^2 =\frac{\mu_P}{6} +\frac{\mu_{P^{\circ}}}{6} 
\end{equation}
where $P^{\circ}$ is the polar polytope of $P$.
\end{theorem}
\begin{proof} 
Notice first that, because of the Fano assumption, the support function of the $\qit$-Cartier divisor $K_X$ is equal to the Newton function 
of $P$ and
thus $\rho^* (-K_X)=\sum_{i=1}^r \nu_i D_i $
since $\rho^* (-K_X)$ and $-K_X$ have the same support function. 
We shall show that 
\begin{equation}\label{eq:RhoChapeauEgalRho}
 \widehat{\mu}_P =\rho^* (-K_X)\rho^* (-K_X)
\end{equation}
Because $(\rho^* (-K_X))^2 =(-K_X)^2 =\mu_{P^{\circ}}$ (for the first equality see \cite[Lemma 13.4.2]{CLS} and for the second one see the $\qit$-Cartier
version of \cite[Theorem 13.4.3]{CLS}), equation (\ref{eq:VarianceMuPolytopePolaireEq}) will follow from 
Theorem \ref{theo:VarianceSpectreGeometrique}.
 By Proposition \ref{prop:MuPDim2}, we have 
\begin{equation}\nonumber
\widehat{\mu}_P =\sum_{i=1}^r (D_i^2 +\frac{\nu_i}{\nu_{i+1}} 
+\frac{\nu_{i+1}}{\nu_{i}})=\sum_{i=1}^r D_i^2 +\sum_{i=1}^r (\frac{\nu_{i-1}}{\nu_{i}} 
+\frac{\nu_{i+1}}{\nu_{i}})
\end{equation}
and, as noticed above, 
\begin{equation}\nonumber
 \rho^* (-K_X)\rho^* (-K_X)=(\sum_{i=1}^r \nu_{i} D_i  )^2 
 =\sum_{i=1}^r \nu_{i}^2 D_i^2 +\sum_{i=1}^r (\nu_i \nu_{i+1} +\nu_i \nu_{i-1})
\end{equation}
so (\ref{eq:RhoChapeauEgalRho}) reads
\begin{equation}\label{eq:RhoChapeauEgalRhoInterm}
\sum_{i=1}^r (\nu_{i}^2 -1)D_i^2 =
\sum_{i=1}^r (\nu_{i}^2 -1)(-\frac{\nu_{i-1}}{\nu_{i}} -\frac{\nu_{i+1}}{\nu_{i}}).
\end{equation}
Notice the following:
\begin{itemize}
 \item if $v_{i-1}$, $v_{i}$ and $v_{i+1}$ are primitive generators of rays of $Y$ inside the same cone of the fan of $X$, we have
 \begin{equation}\nonumber
\nu (v_{i-1}+v_{i+1})=(\frac{\nu_{i-1}}{\nu_{i}} +\frac{\nu_{i+1}}{\nu_{i}})\nu (v_{i})
\end{equation} 
  because $\nu (v_{i-1})=\nu_{i-1}$, $\nu (v_{i})=\nu_{i}$ and $\nu (v_{i+1})=\nu_{i+1}$ and the Newton function is linear on 
  each cone of the fan of $X$. 
  Because $Y$ is smooth and complete, it follows that
 \begin{equation}\nonumber
 v_{i-1} +v_{i+1}=(\frac{\nu_{i-1}}{\nu_{i}} +\frac{\nu_{i+1}}{\nu_{i}})v_i
\end{equation}
 and we get $D_{i}^{2}=-\frac{\nu_{i-1}}{\nu_{i}} -\frac{\nu_{i+1}}{\nu_{i}}$.
 \item Otherwise, $\nu_{i}=1$ 
 due to the Fano condition.
\end{itemize}
Equation (\ref{eq:RhoChapeauEgalRhoInterm}), hence equation (\ref{eq:RhoChapeauEgalRho}), follows from these two observations. 
\end{proof}

\begin{example}\label{ex:EPP}
Let $P$ be the convex hull of $(1,0)$, $(-\lambda_1 ,-\lambda_2 )$ and $(0,1)$ 
 where $\lambda_{1}$ and $\lambda_{2}$ are relatively prime integers. Then
$$\sum_{i=1}^{\mu_P}(\beta_i -1)^2 =\frac{1}{6}[(1+\lambda_{1}+\lambda_{2})+\frac{(1+\lambda_{1}+\lambda_{2})^2}{\lambda_{1}\lambda_{2}}]$$
Notice that $P$ is the polytope of the weighted projective space $\ppit (1,\lambda_{1} ,\lambda_{2})$.
\end{example}

\begin{remark}
Theorem \ref{theo:VarianceMuPolytopePolaire} is not true if we drop the Fano assumption: for instance, if $P$ is the polytope of $\ppit (1,2,2)$, we have $\widehat{\mu}_P =10$  
and $\mu_{P^{\circ}}=\frac{25}{4}$. 
Also, it follows from Proposition \ref{prop:MuPDim2} that
$\widehat{\mu}_{\ell P}=\widehat{\mu}_{P}$ if $\ell$ is an integer greater or equal than one, and thus $\widehat{\mu}_{P}$ can't be seen 
as a volume in general. 
\end{remark}

Finally, we get the following statement for singularities:

\begin{corollary}\label{coro:ConjFanoPolytope}
Let $f$ be a nondegenerate and convenient Laurent polynomial on $(\cit^* )^2$ with 
spectrum at infinity $\alpha_1 , \cdots , \alpha_{\mu}$.
Assume that the Newton polytope $P$ of $f$  is Fano. Then
\begin{equation}\label{eq:VarianceMuPolytopePolaireSing}
\sum_{i=1}^{\mu}(\alpha_i -1)^2 =\frac{\mu_P}{6} +\frac{\mu_{P^{\circ}}}{6} 
\end{equation}
where $P^{\circ}$ is the polar polytope of $P$.
In particular, $\frac{1}{\mu}\sum_{i=1}^{\mu}(\alpha_i -1)^2 \geq \frac{1}{6}$.\qed
\end{corollary}

\section{Hertling's conjecture for regular functions}

\label{sec:Conj}

From Theorem \ref{theo:VarianceMiroirToriqueChampetre} we get:

\begin{proposition}
Let $f$ be a convenient and nondegenerate Laurent polynomial and let $P$ be 
its Newton polytope (assumed to be simplicial). 
Assume that $\widehat{\mu}_P\geq 0$. 
Then
\begin{equation}\label{eq:VarianceSpectreInfini}
\frac{1}{\mu}\sum_{i=1}^{\mu}(\alpha_i -\frac{n}{2})^2 \geq \frac{n}{12}
\end{equation}
were $\mu$ is the global Milnor number of $f$ and $\Spec_{f} (z)=\sum_{i=1}^{\mu}z^{\alpha_i}$ is its spectrum at infinity.\qed
\end{proposition}

\noindent Corollary \ref{coro:VarianceDim2Vraie}, examples of Section \ref{ex:VarianceEPP} and Corollary \ref{coro:ConjFanoPolytope} give some cases for which $\widehat{\mu}_P\geq 0$ and we expect that it is a general rule. 
Notice that, if true, inequality (\ref{eq:VarianceSpectreInfini}) is the best possible: 
for instance, in the situation of Example \ref{ex:EPP}, we have
$$\frac{1}{\mu}\sum_{i=1}^{\mu}(\alpha_i -1)^2 =\frac{1}{6} +\frac{(1+\lambda_{1}+\lambda_{2})}{6\lambda_{1}\lambda_{2}}$$
and the last term on the right can be as small as possible.

Equation (\ref{eq:VarianceSpectreInfini}) motivates
the following conjecture (by Proposition \ref{prop:SpectreAlg01}, we have $\alpha_1 =0$ and $\alpha_{\mu}=n$ if $f$ is a convenient and nondegenerate Laurent polynomial) which has been already stated without any further comments in 
\cite[Remark 4.15]{DoSa1} as a global counterpart of C. Hertling's conjecture for germs of holomorphic functions 
(see \cite{Her}, where the equality is inversed). The tameness assumption is discussed in Section \ref{sec:K}.\\

\noindent {\bf Conjecture on the variance of the spectrum (global version)}
 {\em Let $f$ be a regular, tame function  on a smooth $n$-dimensional affine variety $U$. 
 Then 
\begin{equation}\label{eq:Conj} 
\frac{1}{\mu}\sum_{i=1}^{\mu }(\alpha_{i}-\frac{n}{2})^{2}\geq\frac{1}{12}(\alpha_{\mu }-\alpha_{1})
\end{equation}
where $\alpha_{1}\leq \cdots \leq\alpha_{\mu}$ is the ordered spectrum  of $f$ at infinity.}
\qed \\

\noindent  This is another story, but one should expect equality in formula (\ref{eq:Conj})   
if $f$ belongs to the ideal generated by its partial derivatives (this is the case for quasi-homogeneous polynomials, see \cite{Dimca} and \cite{Her}) because mirror symmetry predicts that
the multiplication by $f$ on its Jacobi ring corresponds to the cup-product by $c_{1}(X)$ on the cohomology algebra of the mirror partner $X$ of $f$.

\end{document}